\documentclass[12pt]{article}
\usepackage{geometry} 
\usepackage{amssymb}
\usepackage{amsmath}
\usepackage{mathrsfs}
\usepackage{amsthm}

\theoremstyle{proclaim}
\newtheorem{theorem}{Theorem}[section]
\newtheorem{lemma}[theorem]{Lemma}
\newtheorem{sublemma}[theorem]{Sublemma}
\newtheorem{corollary}[theorem]{Corollary}
\newtheorem{proposition}[theorem]{Proposition}

\newtheorem{problem}[theorem]{Problem}
\theoremstyle{remark}
\newtheorem{remark}[theorem]{Remark}
\newtheorem{definition}[theorem]{Definition}
\newtheorem{construction}[theorem]{Construction}
\newtheorem{example}[theorem]{Example}
\newtheorem{question}[theorem]{Question}

\newcommand{\nr}[1]{\vspace{0.1ex}\noindent\hspace*{9mm}\llap{\textup{(#1)}}}
 
\newcommand{\norm}[1]{\left\Vert#1\right\Vert}
\newcommand{\abs}[1]{\left\vert#1\right\vert}
\newcommand{\set}[1]{\left\{#1\right\}}
\newcommand{\field}{\mathbb K}
\newcommand{\nat}{\mathbb N}
\newcommand{\real}{\mathbb R}

\newcommand{\eps}{\varepsilon}
\newcommand{\To}{\longrightarrow}
\newcommand{\cball}[1]{B_{#1}}
\newcommand{\identityop}[1]{I_{#1}}
\newcommand{\dual}{\sp{\ast}}
\newcommand{\functional}[2]{\langle #1 ,\,  #2 \rangle }
\newcommand{\finiteop}{\ensuremath{\mathscr{F}}}

\newcommand{\compactop}{\ensuremath{\mathscr{K}}}
\newcommand{\szlenkop}[1]{\ensuremath{\mathscr{S\hspace{-0.8mm}Z}_{\hspace{-1mm}#1}}}
\newcommand{\wscop}{\ensuremath{\mathscr{M}}}
\newcommand{\asplundop}{\ensuremath{\mathscr{D}}}
\newcommand{\sepop}{\ensuremath{\mathscr{X}}}
\newcommand{\wcompactop}{\ensuremath{\mathscr{W}}}
\newcommand{\allop}{\ensuremath{\mathscr{B}}}
\newcommand{\fac}[1]{\mathscr{G}_{#1}}
\newcommand{\cfac}[1]{\overline{\mathscr{G}}_{ \hspace{-0.8mm} #1}}
\newcommand{\opideal}{\ensuremath{\mathscr{I}}}
\newcommand{\mzlenkop}[1]{\ensuremath{\mathscr{M\hspace{-0.8mm}Z}_{\hspace{-1mm}#1}}}
\newcommand{\nzlenkop}[1]{\ensuremath{\mathscr{N\hspace{-1.0mm}Z}_{\hspace{-1mm}#1}}}
\newcommand{\mszlenk}[1]{{\rm Sz}(#1)}

\newcommand{\meeszlenk}[2]{{\rm Sz}_{#1} ( #2 )}
\newcommand{\mdset}[3]{s^{#1}_{#2}( #3 )}
\newcommand{\diam}{{\rm diam}}
\newcommand{\ord}{{\textsc{Ord}}}
\newcommand{\ban}{{\textsc{Ban}}}
\newcommand{\sep}{{\textsc{Sep}}}
\newcommand{\szl}{{\textsc{Szl}}}
\newcommand{\pzl}{{\textsc{Pzl}}}
\newcommand{\cf}[1]{cf \left( #1\right)}
\newcommand{\inj}{^{inj}}
\newcommand{\sur}{^{sur}}
\newcommand{\fin}{^{ < \, \infty}}

\newcommand{\eff}{\mathcal{F}}
\newcommand{\gee}{\mathcal{G}}
\newcommand{\ess}{\mathcal{S}}
\newcommand{\arr}{\mathcal{R}}
\newcommand{\iop}{\mathscr{I}}
\newcommand{\jop}{\mathscr{J}}
\DeclareMathOperator{\Space}{Space}
\DeclareMathOperator{\Op}{Op}

\newcommand{\fresh}{Fr{\'e}chet}

\newcommand{\hajek}{H{\'a}jek}

\newcommand{\pelch}{Pe{\l}czy{\'n}ski}
\newcommand{\reinov}{Re{\u\i}nov}
\numberwithin{theorem}{section}
\numberwithin{equation}{section}

\begin{document}
\title{Asplund operators and the {S}zlenk index\footnote{Research supported by an ANU PhD Scholarship. The present work forms part of the author's doctoral dissertation, written at the Australian National University under the supervision of Dr. Richard J. Loy.}}
\author{Philip A.H. Brooker}
\date{}
\maketitle

\begin{abstract}
For $\alpha$ an ordinal, we investigate the class $\szlenkop{\alpha}$ consisting of all operators whose Szlenk index is an ordinal not exceeding $\omega^\alpha$. We show that each class $\szlenkop{\alpha}$ is a closed operator ideal and study various operator ideal properties for these classes. The relationship between the classes $\szlenkop{\alpha}$ and several well-known closed operator ideals is investigated and quantitative factorization results in terms of the Szlenk index are obtained for the class of Asplund operators.
\end{abstract}

\section*{Introduction}For Banach spaces, the Szlenk index is an isomorphic invariant introduced by W. Szlenk in \cite{Szlenk1968}, where an ordinal-valued index is used to show that there is no separable reflexive Banach space containing all separable reflexive Banach spaces isomorphically. Since then, the Szlenk index has found various applications in the study of the geometry of Banach spaces. For example, it has proved to be useful in the study of universality problems, linear classification of separable $C(K)$ spaces, renorming theory and the Lipschitz and uniform classification of Banach spaces. We refer the reader to \cite{Lancien2006} for a survey on the Szlenk index and its applications in the study of the geometry of Banach spaces. Quite recently, the Szlenk index has also found application in fixed point theory \cite{Dom'inguez2010}, and connections between the Szlenk index, metric embeddings of trees into Banach spaces and the uniform classification of Banach spaces are established in \cite{Baudier}.

The notion of Szlenk index of a Banach space has a natural analogue for operators, and this more general setting for the Szlenk index has been considered by several authors, for example in \cite{Alspach1978}, \cite{Alspach1997}, \cite[p.68]{Bossard1997}, \cite{Bourgain1979} and \cite{Gasparis2005}. A survey on the applications of the Szlenk index to the study of operators on spaces of continuous functions can be found in \cite{Rosenthal2003}. 

The last couple of decades have bore witness to substantial interest in the relationship between the geometry of a Banach space $E$, on the one hand, and the closed ideal structure of $\allop (E)$, on the other ($\allop (E)$ is the Banach algebra of all bounded linear operators $E\To E$). One of the main tools in the study of these relationships is the notion of a closed operator ideal. Given the increasingly important role that the Szlenk index plays in the study of Banach space geometry, we are thus prompted to consider whether there are closed operator ideals naturally associated with the notion of Szlenk index of an operator. We show here that the Szlenk index gives rise to a family of closed operator ideals $\szlenkop{\alpha}$, where $\alpha$ is an ordinal. We study the operator ideal properties of the classes $\szlenkop{\alpha}$ and the relationship of the classes $\szlenkop{\alpha}$ with several other operator ideals already familiar to analysts.
 
We now outline the contents and layout of the current paper. Section~\ref{greatfamily} contains most of the necessary notation and background results that we shall require. In Section~\ref{opsect} we formally introduce the classes $\szlenkop{\alpha}$, establishing them as closed operator ideals and investigating their relationship with the operator ideals of compact operators, Asplund operators and separable range operators. Section~\ref{yemem} is a discussion of some examples involving a number of well-known Banach spaces. In Section~\ref{feekwhack} we show that every $\alpha$-Szlenk operator factors through a Banach space of Szlenk index not exceeding $\omega^{\alpha +1}$. We go on to deduce that for a proper class of ordinals $\alpha$, $\szlenkop{\alpha}$ possesses the factorization property. Section~\ref{theredethere} is then devoted to establishing a similar, but negative, result. In particular, we show that for a proper class of ordinals $\alpha$, $\szlenkop{\alpha}$ lacks the factorization property. In Section~\ref{spacesect} we introduce and study a class of space ideals that are of interest in determining whether the operator ideals $\szlenkop{\alpha +1}$ have the factorization property. We conclude in Section~\ref{closeone} with discussion of possible future directions for work related to the problems addressed here.

Throughout, we rely heavily on results and techniques developed in \cite{Brookera}, where a detailed analysis of the behaviour of the Szlenk index under direct sums is carried out. Indeed, forming direct sums of Banach spaces and their operators is important to many of the results presented here. We also note that the results of Section~\ref{feekwhack} in particular make significant use of the interpolation techniques developed by S.~Heinrich in \cite{Heinrich1980}.

\section{Preliminaries}\label{greatfamily}
\subsection{\textsc{Notation and terminology}}
The class of all Banach spaces is denoted $\ban$, and typical elements of $\ban$ are denoted by the letters $D, \, E, \, F$ and $ G$. For a Banach space $E$ and nonempty bounded $S\subseteq E$, we define $\abs{S} := \sup_{x\in S}\norm{x}$. The closed unit ball of $E$ is denoted $\cball{E}$, and the identity operator of $E$ is $I_E$. By an \emph{operator} we mean a norm-continuous linear map acting between Banach spaces. The class of all operators between arbitrary Banach spaces is denoted $\allop$, and for given Banach spaces $E$ and $F$ the set of all operators $E\To F$ is $\allop (E, \, F)$. For a Banach space $F$, the \emph{canonical embedding of $F$} is the map $\mathfrak{J}_F : F \longrightarrow \ell_\infty (\cball{F\dual})$ given by setting $\mathfrak{J}_F(y) = (\functional{y\dual}{y})_{y\dual \in \cball{F\dual}}$, $y\in F$. For a Banach space $E$, the \emph{canonical surjection onto $E$} is the mapping $\mathfrak{Q}_E : \ell_1(\cball{E}) \longrightarrow E : (a_x)_{x\in \cball{E}} \mapsto \sum_{x\in \cball{E}} a_xx$. 

We write $\ord$ for the class of all ordinals, whose elements shall typically be denoted by the lower-case Greek letters $\alpha$, $\beta$ and $\gamma$. For an ordinal $\alpha$, we write $\cf{\alpha}$ for the cofinality of $\alpha$. For $\Lambda$ a set, $\Lambda\fin$ shall denote the set of all nonempty finite subsets of $\Lambda$. Whenever $\Lambda$ and $\Upsilon$ are used to denote index sets over which we take direct sums and direct products, we assume for simplicity that $\Lambda$ and $\Upsilon$ are nonempty.

Let $1\leqslant q<\infty$. We say that $p \in \set{0} \cup [1, \, \infty )$ is \emph{predual} to $q$ if it satisfies:
\begin{equation*}
p =
\begin{cases}
0& \text{if $q=1$},\\
{q}{(q-1)^{-1}}& \text{if $1<q<\infty$.}
\end{cases}
\end{equation*}

For $1\leqslant p\leqslant \infty$, a set $\Lambda$ and Banach spaces $E_\lambda$, $\lambda \in \Lambda$, the $\ell_p$-direct sum of $\set{E_\lambda \mid \lambda \in \Lambda}$ is denoted $(\bigoplus_{\lambda \in \Lambda}E_\lambda)_p$, and the $c_0$-direct sum of $\set{E_\lambda \mid \lambda \in \Lambda}$ is denoted $(\bigoplus_{\lambda \in \Lambda}E_\lambda)_0$. Throughout, for $1<p, q<\infty$ satisfying $p+q=pq$, we implicitly identify $(\bigoplus_{\lambda \in \Lambda}E_\lambda)\dual_p$ with $(\bigoplus_{\lambda \in \Lambda}E_\lambda\dual)_q$, so that the dual of a direct sum is the dual direct sum of the duals of the spaces $E_\lambda$. Making this identification allows us to consider direct products of the form $\prod_{\lambda \in \Lambda}K_\lambda$, where $K_\lambda \subseteq E_\lambda\dual$ and $(\abs{K_\lambda})_{\lambda \in \Lambda} \in \ell_q (\Lambda)$, as subsets of $(\bigoplus_{\lambda \in \Lambda}E_\lambda)\dual_p$. Similarly, $(\bigoplus_{\lambda \in \Lambda}E_\lambda)\dual_0$ is naturally identified with $(\bigoplus_{\lambda \in \Lambda}E_\lambda\dual)_1$ throughout.

We shall often consider operators $T: (\bigoplus_{\lambda \in \Lambda}E_\lambda)_p  \To (\bigoplus_{\upsilon \in \Upsilon}F_\upsilon)_p$, where $\Lambda$ and $\Upsilon$ are sets, $\set{E_\lambda \mid \lambda \in \Lambda}$ and $\set{F_\upsilon \mid \upsilon \in \Upsilon}$ families of Banach spaces and $p=0$ or $1<p<\infty$. In this setting, for $\arr \subseteq \Lambda$ we denote by $U_\arr$ the canonical injection of $(\bigoplus_{\lambda \in \arr}E_\lambda)_p $ into $(\bigoplus_{\lambda \in \Lambda}E_\lambda)_p$. For $\ess \subseteq \Upsilon$, we denote by $V_\ess$ the canonical injection of $(\bigoplus_{\upsilon \in \ess}F_\upsilon)_p $ into $(\bigoplus_{\upsilon \in \Upsilon}F_\upsilon)_p$, and by $Q_\ess$ the canonical surjection of $(\bigoplus_{\upsilon \in \Upsilon}F_\upsilon)_p$ onto $(\bigoplus_{\upsilon \in \ess}F_\upsilon)_p$. Thus $V_\ess$ and $Q_\ess$ act to and from the codomain of $T$ respectively.

We work within the theory of operator ideals as expounded by A.~Pietsch in \cite{Pietsch1980}. The starting point of this theory is the following definition that we shall refer to in the proof of Theorem~\ref{idealthm}.
\begin{definition}\label{idealdefn}(\cite[\S1.1.1]{Pietsch1980})
An \emph{operator ideal} $\opideal$ is a subclass of $\allop$ such that for Banach spaces $E$ and $F$, the components $ \opideal (E, \, F ) :=  \allop (E, \, F) \cap \opideal $ satisfy the following three conditions:

\nr{${\rm OI}_1$} $\identityop{\field} \in \opideal$;

\nr{${\rm OI}_2$} $S + T \in \opideal (E, \, F) $ whenever $S,\, T \in \opideal (E, \, F)$.

\nr{${\rm OI}_3$} $U\in \allop (D, \, E)$, $T\in \opideal (E, \, F) $ and $V \in \allop (F, \, G)$ implies $VTU \in \opideal$.
\end{definition}

We otherwise assume the reader is familiar with the rudiments of operator ideal theory, and refer the reader to \cite[Part~I]{Pietsch1980} for any unexplained notions regarding operator ideals. In particular, we assume the reader is familiar with what it means for an operator ideal to be closed, injective and surjective. For a given operator ideal $\opideal$, the closed, injective and surjective hulls of $\opideal$ are denoted $\overline{\opideal}$, $\opideal\inj$ and $\opideal\sur$, respectively. We also assume knowledge of basic notions and facts regarding space ideals \cite[p.53]{Pietsch1980}.

Well-known operator ideals that we shall be concerned with here are the compact operators $\compactop$, the weakly compact operators $\wcompactop$, the separable range operators $\sepop$ and the Hilbert space-factorable operators $\Gamma_2$. For a Cartesian Banach space $E$ (that is, $E$ is isomorphic to its square $E\oplus E$), we denote by $\fac{E}$ the operator ideal consisting of all operators that admit a continuous linear factorization through $E$.

For an operator ideal $\iop$, we denote by $\Space (\iop )$ the space ideal consisting of all Banach spaces whose identity operator belongs to $\iop$. For a space ideal $\mathtt{I}$, we denote by $\Op (\mathtt{I})$ the operator ideal consisting of all operators that admit a continuous linear factorization through an element of $\mathtt{I}$. For operator ideals $\iop$ and $\jop$, we say that $\iop$ has the \emph{$\jop$-factorization property} if $\iop \subseteq \Op (\Space(\jop))$; evidently, this implies that $\iop \subseteq \jop$. An operator ideal $\iop$ has the \emph{factorization property} if it has the $\iop$-factorization property.

In various parts of the paper we call upon a factorization result due to S.~Heinrich. In order to state Heinrich's result, we require the following definition.

\begin{definition}
Let $\iop$ and $\jop$ be operator ideals and $1<p<\infty$. We say that $(\iop , \, \jop)$ is a \emph{$\Sigma_p$-pair} if the following holds for any sequences of Banach spaces $(E_m)_{m\in \nat}$ and $(F_n)_{n\in\nat}$ and $T\in \allop ((\bigoplus_{m \in \nat}E_m)_p , \, (\bigoplus_{n\in\nat}F_n)_p )$: if $Q_\gee TU_\eff \in \iop$ for all $\eff, \, \gee \in \nat\fin$, then $T\in \jop$.
\end{definition}

Heinrich establishes the following result in \cite{Heinrich1980}:

\begin{theorem}\label{manna}
Let $1<p<\infty$ and let $\iop$ and $\jop$ be surjective operator ideals such that $(\iop , \, \jop)$ is a $\Sigma_p$-pair and $\jop$ is injective. Then $\iop$ has the $\jop$-factorization property.
\end{theorem}

We note that Theorem~\ref{manna} is presented and proved in \cite{Heinrich1980} under the additional hypothesis that $\iop = \jop$. This restriction is, in fact, unnecessary, and we leave it to the interested reader to verify that Heinrich's proof of Theorem~\ref{manna} holds in the generality in which it is stated above (a straightforward notational substitution in Heinrich's proofs should suffice for the reader familiar with interpolation theory).

A real Banach space $E$ is said to be \emph{Asplund} if every real-valued convex continuous function defined on a convex open subset $U$ of $E$ is {\fresh} differentiable on a dense $G_\delta$ subset of $U$. A complex Banach space $E$ is said to be Asplund if its underlying real Banach space $E_\real$ is Asplund in the real scalar sense. Of particular importance to the context of our discussion is the following theorem that collects several useful characterizations of Asplund spaces; for $C\subseteq E\dual$, $\eps>0$ and $x\in E$, the \emph{$w\dual$-slice of $C$ determined by $x$ and $\eps$} is the set $\set{x\dual \in C\mid \Re\functional{x\dual}{x} > \sup \set{\Re \functional{y\dual}{x}\mid y\dual \in C}-\eps} $.

\begin{theorem}\label{Aspequiv}
Let $E$ be a Banach space. The following are equivalent:

\nr{i} $E$ is an Asplund space;

\nr{ii} Every separable subspace of $E$ is an Asplund space;

\nr{iii} Every separable subspace of $E$ has separable dual;

\nr{iv} Every bounded nonempty subset of $E\dual$ admits nonempty $w\dual$-slices of arbitrarily small diameter.
\end{theorem}

Theorem~\ref{Aspequiv} is proved for real Banach spaces in Chapter~I.5 of \cite{Deville1993}. For complex Banach spaces $E$, Theorem~\ref{Aspequiv} follows from the real scalar case and properties of the canonical linear surjection $\varphi: x\dual \mapsto \Re x\dual$ of $E\dual$ onto $(E_\real)\dual$. In particular, $\varphi$ is a norm-to-norm isometric, $\sigma(E\dual, \, E)$-to-$\sigma ((E_\real)\dual, \, E_\real)$ homeomorphism; this is easily deduced from \cite[Proposition~1.9.3]{Megginson1998}.

Let $E$ and $F$ be Banach spaces. An operator $T:E\To F$ is \emph{Asplund} if for any finite positive measure space $(\Omega, \, \Sigma , \, \mu)$, any $S\in \allop (F, \,  L_\infty (\Omega, \, \Sigma , \, \mu))$ and any $\eps >0$, there exists $B\in \Sigma$ such that $\mu (B)>\mu (\Omega )-\eps$ and $\set{f\chi_B\mid f\in ST(\cball{E})}$ is relatively compact in $L_\infty (\Omega, \, \Sigma , \, \mu)$ (here $\chi_B$ denotes the characteristic function of $B$ on $\Omega$). The class of all Asplund operators is denoted $\asplundop$. We note that some authors, for example in \cite{Pietsch1980} and \cite{Heinrich1980}, refer to Asplund operators as \emph{decomposing operators}. Standard references for Asplund operators are \cite{Pietsch1980} and \cite{Stegall1981}, where it is shown that the Asplund operators form a closed operator ideal and that a Banach space is an Asplund space if and only if its identity operator is an Asplund operator. A further result is that every Asplund operator factors through an Asplund space, due independently to O.~{\reinov} \cite{Reuinov1978}, S.~Heinrich \cite{Heinrich1980} and C.~Stegall \cite{Stegall1981}.

\subsection{The Szlenk index}

We now define the Szlenk index, noting that our definition varies from that given by W. Szlenk in \cite{Szlenk1968}. However, the two definitions give the same index for operators acting on separable Banach spaces containing no isomorphic copy of $\ell_1$ (see the proof of \cite[Proposition~3.3]{Lancien1993} for details).  

Let $E$ be a Banach space, $K \subseteq E\dual$ a $w\dual$-compact set and $\eps >0$. Define
\begin{equation*}
\mdset{\mbox{}}{\eps}{K} = \set{x \in K \mid \diam (K \cap V)> \eps \mbox{ for every } w\dual \mbox{-open }V\ni x}.
\end{equation*} We iterate $s_\eps$ transfinitely as follows: $\mdset{0}{\eps}{K} = K$, $ \mdset{\alpha+1}{\eps}{K}= \mdset{\mbox{}}{\eps}{\mdset{\alpha}{\eps}{K}}$ for each ordinal $\alpha$ and $ \mdset{\alpha}{\eps}{K} = \bigcap_{\beta< \alpha} \mdset{\beta}{\eps}{K}$ whenever $\alpha $ is a limit ordinal. 

The \emph{$\eps$-Szlenk index of $K$}, denoted $\meeszlenk{\eps}{K}$, is the class of all ordinals $\alpha$ such that $\mdset{\alpha}{\eps}{K} \neq \emptyset$. The \emph{Szlenk index of $K$} is the class $\bigcup_{\eps >0}\meeszlenk{\eps}{K}$. Note that $\meeszlenk{\eps}{K}$ (resp., $\mszlenk{K}$) is either an ordinal or the class $\ord$ of all ordinals. If $\meeszlenk{\eps}{K}$ (resp., $\mszlenk{K}$) is an ordinal, then we write $\meeszlenk{\eps}{K} < \infty$ (resp., $\mszlenk{K}<\infty$), and otherwise we write $\meeszlenk{\eps}{K} = \infty$ (resp., $\mszlenk{K}=\infty$). For a Banach space $E$, the \emph{$\eps$-Szlenk index of $E$} is $\meeszlenk{\eps}{E} = \meeszlenk{\eps}{{\cball{E\dual}}}$, and the \emph{Szlenk index of $E$} is $\mszlenk{E} = \mszlenk{\cball{E\dual}}$. If $T: E\To F$ is an operator, the \emph{$\eps$-Szlenk index of $T$} is $\meeszlenk{\eps}{T} = \meeszlenk{\eps}{{T\dual\cball{F\dual}}}$, whilst the \emph{Szlenk index of $T$} is $\mszlenk{T} = \mszlenk{T\dual\cball{F\dual}}$. For $\alpha$ an ordinal, $\szl_\alpha := \set{E\in\ban \mid \mszlenk{E}\leqslant \omega^\alpha}$.

It is clear that the Szlenk index of a nonempty $w\dual$-compact set cannot be $0$. We also note that, by $w\dual$-compactness, the $\eps$-Szlenk index of a nonempty $w\dual$-compact set $K$ is never a limit ordinal.

The following proposition collects some known facts about Szlenk indices.
\begin{proposition}\label{collection}
Let $E$ and $F$ be Banach spaces.

\nr{i} If $E$ is isomorphic to a quotient or subspace of $F$, then $ \mszlenk{E}\leqslant \mszlenk{F}$. In particular, the Szlenk index is an isomorphic invariant of a Banach space.

\nr{ii} $\mszlenk{E} < \infty$ if and only if $E$ is Asplund.

\nr{iii} If $K\subseteq E\dual$ is nonempty, absolutely convex and $w\dual$-compact, then either $\mszlenk{K} = \infty$ or there exists an ordinal $\alpha$ such that $\mszlenk{K} = \omega^{\alpha}$. In particular, for $T\in\allop$ either $\mszlenk{T} = \infty$ or $\mszlenk{T}=\omega^\alpha$ for some ordinal $\alpha$.

\nr{iv} If $E$ is separable, then $E\dual$ is norm separable if and only if $\mszlenk{E}< \omega_1$, if and only if $\mszlenk{E} < \infty$.

\nr{v} $\mszlenk{E\oplus F} = \max \set{\mszlenk{E}, \, \mszlenk{F}}$.

\nr{vi} $\szl_\alpha$ is a space ideal for each ordinal $\alpha$.
\end{proposition}

We briefly indicate the origins of the various assertions of Proposition~\ref{collection}. Part (i) is well-known; see, for example, \cite[p.2032]{H'ajek2007a}. Part (ii) follows from Theorem~\ref{Aspequiv}(i)$\iff$(iv) above. Part (iii) is due to G.~Lancien \cite{Lancien1996}; note that although Lancien's proof is given for the case where $K$ is the closed unit ball of a dual Banach space, his argument works equally well in the more general setting presented above. We mention also that the first occurrence of a statement like (iii) is a similar result for the Lavrientiev index of a Banach space due to A.~Sersouri \cite{Sersouri1989}. For (iv), see \cite[Theorem~3.1]{Odell2004} and its proof. Part (v) follows from Lemma~\ref{carbon} of the current paper (which is due to P.~{\hajek} and G.~Lancien \cite{H'ajek2007}). Finally, (vi) is a consequence of (i), (v) and the well-known fact that a Banach space is finite-dimensional if and only if it has Szlenk index equal to $1$ (this is noted in \cite[p.211]{Lancien2006}, but see also Proposition~\ref{lastlabel} below).

\section{$\alpha$-Szlenk operators}\label{opsect}
Here we consider the Szlenk index of an operator and show that this index can be used in a natural way to define a class of closed operator ideals indexed by the class of all ordinals. 
\begin{definition}
For each ordinal $\alpha$, define $\szlenkop{\alpha} := \set{T \in \allop \mid \mszlenk{T} \leqslant \omega^\alpha}$. An element of $\szlenkop{\alpha}$ shall be known as an \emph{$\alpha$-Szlenk operator}. For each ordinal $\alpha$ and pair of Banach spaces $(E, \, F)$, define $\szlenkop{\alpha} (E, \, F) := \allop (E, \, F)\cap \szlenkop{\alpha} $.
\end{definition}

It is trivial that $\szlenkop{\alpha} \subseteq \szlenkop{\beta}$ whenever $\alpha$ and $\beta$ are ordinals satisfying $\alpha \leqslant \beta$. In fact, $\szlenkop{\alpha} \subsetneq \szlenkop{\beta}$ whenever $\alpha <\beta$. Indeed, it is shown in \cite[Proposition~2.16]{Brookera} that for an each ordinal $\alpha$ there exists a Banach space $E$ with $\mszlenk{E} = \omega^{\alpha+1}$; the identity operator of such a space $E$ belongs to $\szlenkop{\alpha +1} \backslash \szlenkop{\alpha}$.

The following theorem is the main result of the current section.
\begin{theorem}\label{idealthm}
For $\alpha$ an ordinal, $\szlenkop{\alpha}$ is a closed, injective and surjective operator ideal.
\end{theorem}

For $\alpha =0$, the assertion of Theorem~\ref{idealthm} follows from the following proposition and the well-known fact that $\compactop$ is closed, injective and surjective.

\begin{proposition}\label{basecase}
$\szlenkop{0} = \compactop$.
\end{proposition}

Proposition~\ref{basecase} is a consequence of Schauder's theorem and the following general result:

\begin{proposition}\label{lastlabel}
Let $E$ be a Banach space, $K$ a nonempty $w\dual$-compact subset of $E\dual$. Then $K$ is norm-compact if and only if $\mszlenk{K}=1$.
\end{proposition}
\begin{proof}
We use the fact that $K$ is norm-compact if and only if the relative norm and $w\dual$ topologies of $K$ are the same (see, e.g., \cite[Corollary~3.1.14]{Engelking1989}).

First suppose that $\mszlenk{K}=1$. Let $(x_i)_{i\in I}$ be a $w\dual$-convergent net in $K$; the norm-compactness of $K$ will follow if $(x_i)_{i\in I}$ is necessarily norm convergent. Let $x = w\dual -\lim_i x_i \in K$ and note that, as $x\notin \bigcup_{\eps>0}s_\eps (K)$, for every $\eps >0$ there exists $w\dual$-open $U_\eps \ni x$ such that $\diam (U_\eps \cap K)\leqslant \eps$. For each $\eps > 0$ let $j_\eps \in I$ be such that $j_\eps \prec j'$ implies $x_{j'} \in U_\eps \cap K$. Then $j_\eps\prec j'$ implies $\Vert x-x_{j'}\Vert \leqslant \eps$. As $\eps >0$ is arbitrary, $\Vert x-x_i\Vert {\rightarrow} 0$.

Now suppose $\mszlenk{K} >1$. Then there is $x\in K$ and $\eps >0$ such that $x\in s_\eps (K)$, so for each $w\dual$-open $U\ni x$ there is $x_U \in U\cap K$ such that $\norm{x-x_U}>\eps /2$. Since $x_U \stackrel{w\dual}{\rightarrow} x$ and $x_U \stackrel{\norm{\, \cdot \,}}{\nrightarrow}x$ (here, the set of $w\dual$-open sets containing $x$ carries the usual order induced by reverse set inclusion), the relative norm and $w\dual$ topologies of $K$ are not the same. Hence $K$ is not norm-compact.
\end{proof}

We now prove the general case.

\begin{proof}[Proof of Theorem~\ref{idealthm}]
Let $\alpha$ be an ordinal. We must first show that $\szlenkop{\alpha}$ satisfies ${\rm OI}_1$-${\rm OI}_3$ of Definition~\ref{idealdefn}. To see that $\szlenkop{\alpha}$ satisfies ${\rm OI}_1$, note that by Proposition~\ref{basecase} we have
\begin{equation*}
\identityop{\field} \in \compactop = \szlenkop{0} \subseteq \szlenkop{\alpha}.
\end{equation*}

Next we show that $\szlenkop{\alpha}$ satisfies ${\rm OI}_3$. Let $D$, $E$, $F$ and $G$ be Banach spaces and $U \in \allop (D, \, E )$, $T\in \szlenkop{\alpha} (E, \, F)$ and $V \in \allop (F, \, G)$ operators. We want to show that $VTU\in\szlenkop{\alpha}$; this is clearly true if either $U$ or $V$ is zero, so we henceforth assume that $U$ and $V$ are nonzero. It suffices to show separately that $TU \in \szlenkop{\alpha}$ and $VT \in \szlenkop{\alpha}$. The fact that $TU \in \szlenkop{\alpha}$ will be deduced from the following generalization of \cite[Lemma~2]{H'ajek2007a}.
\begin{lemma}\label{poopoopoo}
Let $D$ and $G$ be Banach spaces, $S\in \allop (D, \, G)$ a nonzero operator, $K\subseteq G\dual$ a $w\dual$-compact set, $\alpha$ an ordinal and $\eps > 0$. Then $s^\alpha_\eps (S\dual K) \subseteq S\dual \big(s^\alpha_{\eps /(2\norm{S})}(K)\big) $. \end{lemma}
\begin{proof}
We proceed by induction on $\alpha$. The assertion of the lemma is trivially true for $\alpha =0$. Suppose that $\beta >0$ is an ordinal such that the assertion of the lemma is true for all $\alpha < \beta$; we show that it is then true for $\alpha = \beta$. First suppose that $\beta$ is a successor, say $\beta = \gamma+1$. Let $x \in s^{\beta}(S\dual K)$. Then there is a net $(x_i)_{i\in I}$ in $s^\gamma_\eps (S\dual K)$ with $x_i \stackrel{w\dual}{\rightarrow}x$ and $\norm{x_i - x}>\eps /2$ for all $i$ (for example, let $I$ be the set of all $w\dual$-neighbourhoods of $x$, ordered by reverse set inclusion). By the induction hypothesis, for each $i$ there is $y_i \in s^\gamma_{\eps /(2\norm{S})}(K)$ such that $S\dual y_i = x_i$. Passing to a subnet, we may assume that the net $(y_i)_{i\in I}$ has a $w\dual$-limit $y\in s^\gamma_{\eps /(2\norm{S})}(K)$. Then $S\dual y = x$ and for all $i$ we have $\norm{y_i - y} \geqslant \norm{x_i - x}/ \norm{S} > \eps / (2\norm{S})$, hence $y\in s^{\beta}_{\eps /(2\norm{S})}(K)$. It follows that the assertion of the lemma passes to successor ordinals. 

Now suppose that $\beta$ is a limit ordinal. Let $x\in s^\beta_\eps (S\dual K) = \bigcap_{\alpha <\beta}s^\alpha_\eps (S\dual K)$. For each $\alpha < \beta$ there is $y_\alpha \in s^\alpha_{\eps /(2\norm{S})}(K)$ with $S\dual y_\alpha = x$. The net $(y_\alpha)_{\alpha <\beta}$ admits a subnet $(y_j)_{j\in J}$ with $w\dual$-limit $y\in \bigcap_{\alpha <\beta}s^\alpha_{\eps / (2\norm{S)}}(K) = s^\beta_{\eps /(2\norm{S})}(K)$. Since $S\dual y = x$, we are done.\end{proof}
By Lemma~\ref{poopoopoo},
\begin{equation*}
\mszlenk{TU} = \sup_{\eps>0}\meeszlenk{\eps}{(TU)\dual \cball{F\dual}} \leqslant \sup_{\eps>0}\meeszlenk{\eps / (2\norm{U})}{T\dual \cball{F\dual}} = \mszlenk{T}\leqslant \omega^\alpha ,
\end{equation*} hence $TU\in \szlenkop{\alpha}$.

As $VT = (\norm{V}^{-1}V)T(\norm{V}I_E)$ and $T(\norm{V}I_E)\in\szlenkop{\alpha}$ (take $U=\norm{V}I_E$ above), to show that $VT\in\szlenkop{\alpha}$ we may assume that $\norm{V}\leqslant 1$. Then
\begin{equation*}
(VT)\dual \cball{G\dual} = T\dual (V\dual \cball{G\dual}) \subseteq T\dual \cball{F\dual},
\end{equation*}
hence $\mszlenk{VT} = \mszlenk{(VT)\dual \cball{G\dual}} \leqslant \mszlenk{T\dual \cball{F\dual}} = \mszlenk{T} \leqslant \omega^\alpha$,
as desired. We have now shown that $\szlenkop{\alpha}$ satisfies ${\rm OI}_3$.

To show that $\szlenkop{\alpha}$ satisfies ${\rm OI}_2$, we make use of the following lemma of P. {\hajek} and G. Lancien (see \cite[Equation~(2.3)]{H'ajek2007}). The author is grateful to Professor Lancien for communicating to him a corrected proof of Lemma~\ref{carbon} (the proof of Lemma~\ref{carbon} in \cite{H'ajek2007} seems to be slightly incorrect); the proof of Sublemma~\ref{carbonesque} of the current paper uses some similar arguments.

\begin{lemma}\label{carbon}
Let $E_1, \ldots , E_n$ be Banach spaces and let $K_1 \subseteq E_1\dual, \ldots , K_n \subseteq E_n\dual$ be $w\dual$-compact sets. Consider $\prod_{i=1}^n K_i$ as a subset of $(\bigoplus_{i=1}^n E_i)_1\dual$. Then, for all $\eps > 0$ and ordinals $\alpha$,
\begin{equation}\label{reedideas}
s^{\omega^\alpha}_{\eps}\hspace{-1.0mm}\left( \prod_{i=1}^nK_i\right) \subseteq  \bigcup_{\substack{g_1, \ldots , g_n <\,  \omega \\ g_1 + \ldots + g_n = \,1}} \hspace{0.5mm}\prod_{i=1}^n \mdset{\omega^\alpha \cdot g_i}{\eps}{K_i}.
\end{equation}
\end{lemma}

Let $E$ and $F$ be Banach spaces and let $S, T \in \allop (E, \, F)$ be operators such that $S+T \notin \szlenkop{\alpha}$. Define operators $Q : E \To E\oplus_1 E $ and $R: E\oplus_1 E \To F$ by setting $Qx = (x, \, x)$ for $x\in E$, and $R (y, \, z) = Sy + Tz$ for $(y, \, z) \in E\oplus_1 E$, so that $RQ = S+T \notin \szlenkop{\alpha}$. Then $\mszlenk{Q\dual (R\dual \cball{F\dual})} > \omega^\alpha$, hence $\mszlenk{R\dual \cball{F\dual}} > \omega^\alpha$ since $\szlenkop{\alpha}$ satisfies ${\rm OI}_3$. We have $
R\dual \cball{F\dual} = \set{( S\dual x, \, T\dual x ) \mid x \in \cball{F\dual}} \subseteq S\dual \cball{F\dual} \times T\dual \cball{F\dual},$
hence $\mszlenk{S\dual \cball{F\dual} \times T\dual \cball{F\dual}} > \omega^\alpha$. Let $\eps > 0$ be such that $\mdset{\omega^\alpha}{\eps}{S\dual \cball{F\dual} \times T\dual \cball{F\dual}}$ is nonempty. By Lemma~\ref{carbon}, either $\mdset{\omega^\alpha}{\eps}{S\dual \cball{F\dual}}$ or $\mdset{\omega^\alpha}{\eps}{T\dual \cball{F\dual}}$ is nonempty, hence either $\mszlenk{S} > \omega^\alpha$ or $\mszlenk{T} > \omega^\alpha$. In other words, either $S\notin \szlenkop{\alpha}$ or $T\notin \szlenkop{\alpha}$. Thus $\szlenkop{\alpha}$ satisfies ${\rm OI}_2$, and is an operator ideal.

The injectivity of $\szlenkop{\alpha}$ follows from the fact that for Banach spaces $E$ and $F$ and an operator $T\in \allop (E, \,  F)$, the Szlenk indices of $T$ and $\mathfrak{J}_FT$ are determined by the same set, namely $T\dual\cball{F\dual} = (\mathfrak{J}_FT)\dual \cball{\ell_\infty (\cball{F\dual})\dual}$.

The surjectivity of $\szlenkop{\alpha}$ is only slightly more difficult. Notice that for Banach spaces $E$ and $F$ and $T\in \allop (E, \, F)$, the restriction of $\mathfrak{Q}_E\dual$ to $T\dual \cball{F\dual}$ is a norm-isometric $w\dual$-homeomorphic embedding of $T\dual \cball{F\dual}$ into $\ell_1(\cball{E})\dual$. It follows then that $\mathfrak{Q}_E\dual (s^\alpha_\eps (T\dual \cball{F\dual})) \subseteq s^\alpha_\eps (\mathfrak{Q}_E\dual T\dual \cball{F\dual}) = s^\alpha_\eps ((T\mathfrak{Q}_E)\dual \cball{F\dual})$ for all ordinals $\alpha$ and $\eps >0$ (the proof is a straightforward transfinite induction), hence $\mszlenk{T}\leqslant \mszlenk{T\mathfrak{Q}_E}$. In particular, $\szlenkop{\alpha}$ is surjective.

Finally, we turn our attention to showing that $\szlenkop{\alpha}$ is a \emph{closed} operator ideal. Recall that for a Banach space $E$, a nonempty, $w\dual$-compact set $K\subseteq E\dual$ and $x\in E\dual$, there exists $y\in K$ such that $\norm{x-y} = d(x, \, K)$ (here $d(x, \, K)$ denotes the norm distance of $x$ to $K$, defined as $d(x, \, K):= \inf \set{\norm{x-z}\mid z\in K}$). Our proof that $\szlenkop{\alpha}$ is closed will be a straightforward application of the following lemma.

\begin{lemma}\label{weekendoz}
Let $D$ be a Banach space, $\eps>0$ and $K,\, L \subseteq D\dual$ nonempty, $w\dual$-compact sets with $\sup \set{d(x, \, L)\mid x\in K}\leqslant \eps /8$. Then $\meeszlenk{\eps}{K}\leqslant \meeszlenk{\eps/4}{L}$.
\end{lemma}

\begin{proof}
It clearly suffices to show that for all $\gamma \in  \meeszlenk{\eps}{K}$,
\begin{equation}\label{strudel} s^\gamma_{\eps/4}(L) \neq \emptyset \quad \mbox{and} \quad \sup \big\{d(x, \, s^\gamma_{\eps/4}(L))\mid x\in s^\gamma_{\eps}(K)\big\}\leqslant \eps /8. \end{equation}
The assertions of (\ref{strudel}) hold trivially for $\gamma = 0$. Suppose that $\beta \in \meeszlenk{\eps}{K}$ is such that (\ref{strudel}) holds for all $\gamma<\beta$; we will show that (\ref{strudel}) holds for $\gamma = \beta$.

First suppose that $\beta$ is a successor, say $\beta = \zeta+1$, and let $x\in s^\beta_\eps(K)$. Then there exists a net $(x_i)_{i\in I}$ in $s^\zeta_\eps(K)$ with $x_i\stackrel{w\dual}{\rightarrow}x$ and $\norm{x_i - x}>\eps/2$ for all $i$ (for example, take $I$ to be the set of all $w\dual$-neighbourhoods of $x$, ordered by reverse set inclusion). By the induction hypothesis, for each $i\in I$ there is $y_i \in s^\zeta_{\eps/4}(L)$ with $\norm{x_i - y_i}\leqslant \eps/8$. Passing to a subnet, we may assume $(y_i)_{i\in I}$ has a $w\dual$-limit, $y$ say, in $s^\zeta_{\eps/4}(L)$. By $w\dual$-lower semicontinuity, $\norm{x-y}\leqslant \liminf_{i\in I}\norm{x_i-y_i}\leqslant \eps /8$. Thus, for all $i\in I$,\begin{equation*}  \norm{y - y_i}\geqslant  \norm{x-x_i} - \norm{x_i - y_i}-\norm{x-y}> \frac{\eps}{2}-\frac{\eps}{8}-\frac{\eps}{8} = \frac{\eps}{4}, \end{equation*} hence $y\in s_{\eps/4}(s^\zeta_{\eps/4}(L)) = s^\beta_{\eps /4}(L)$. In particular, $ s^\beta_{\eps /4}(L)$ is nonempty. Moreover, $d(x, \,  s^\beta_{\eps /4}(L))\leqslant \norm{x-y}\leqslant \eps/8$. Thus, since $x\in s^\beta_\eps(K)$ is arbitrary, we conclude that $\sup \big\{d(x, \, s^\beta_{\eps/4}(L))\mid x\in s^\beta_{\eps}(K)\big\}\leqslant \eps /8$. We have now shown that (\ref{strudel}) passes to successor ordinals in $\meeszlenk{\eps}{K}$.

Now suppose that $\beta$ is a limit ordinal. Then $s^\beta_{\eps /4}(L)$ is nonempty by the induction hypothesis and $w\dual$-compactness. For the second assertion of (\ref{strudel}), we again let $x\in s^\beta_\eps(K)$. By the induction hypothesis, for each $\zeta<\beta$ there is $y_\zeta \in s^\zeta_{\eps/4}(L)$ such that $\Vert x-y_\zeta \Vert\leqslant \eps/8$. Let $(z_j)_{j\in J}$ be a $w\dual$-convergent subnet of $(y_\zeta)_{\zeta<\beta}$, with $w\dual$-limit $y$, say. Then $y\in \bigcap_{\zeta<\beta}s^\zeta_{\eps/4}(L) = s^\beta_{\eps/4}(L)$ and $\norm{x-y}\leqslant \liminf_{j\in J}\Vert x-z_j\Vert \leqslant \eps /8$, hence $d(x, \, s^\beta_{\eps/4}(L)) \leqslant \norm{x-y}\leqslant \eps /8$. As $x\in s^\beta_\eps(K)$ is arbitrary, the second assertion of (\ref{strudel}) holds for $\gamma=\beta$. This completes the proof of the lemma.
\end{proof}

Let $E$ and $F$ be Banach spaces and $T\in \allop (E, \, F)$ an operator such that $T\notin \szlenkop{\alpha}$. Then there is $\eps >0$ such that $\meeszlenk{\eps}{T}>\omega^\alpha$. Let $S\in \allop (E, \, F)$ be such that $\norm{T-S}<\eps /8$. Taking $K = T\dual \cball{F\dual}$ and $L = S\dual \cball{F\dual}$ in the statement of Lemma~\ref{weekendoz} yields $\omega^\alpha <\meeszlenk{\eps}{T} \leqslant \meeszlenk{\eps/4}{S}\leqslant \mszlenk{S}$, hence $S\notin \szlenkop{\alpha}$. In particular, the open ball in $\allop (E, \, F) $ centred at $T$ and of radius $\eps /8$ has trivial intersection with $\szlenkop{\alpha}(E, \, F)$. It follows that $\szlenkop{\alpha}(E, \, F)$ is closed in $\allop (E, \, F)$, and the proof of Theorem~\ref{idealthm} is complete.
\end{proof}

We now describe the relationship between the classes $\szlenkop{\alpha}$ and the class of Asplund operators. For this we shall call on the following characterization of Asplund operators that follows readily from work of C.~Stegall, in particular \cite[Proposition~2.10 and Theorem~1.12]{Stegall1981}.

\begin{proposition}\label{chickennuggets}
Let $E$ and $F$ be Banach spaces and $T:E\To F$ an operator. Then $T$ is Asplund if and only if for every separable Banach space $D$ and every operator $S: D\To E$, the set $S\dual T\dual \cball{F\dual}$ is norm separable.
\end{proposition}

We also require the following result concerning metrizable $w\dual$-compact sets; the proof is essentially contained in the proof of Proposition~\ref{collection}(iv).
\begin{lemma}\label{nickenchuggets}
Let $K$ be a $w\dual$-compact set that is metrizable in the $w\dual$ topology and nonseparable in the norm topology. Then $\mszlenk{K}=\infty$.
\end{lemma}

The following proposition asserts that the class of Asplund operators coincides with $\bigcup_{\alpha\in\ord}\szlenkop{\alpha}$.

\begin{proposition}\label{opequiv}
Let $E$ and $F$ be Banach spaces and $T :E\To F$ an operator. The following are equivalent:

\nr{i} $T$ is an $\alpha$-Szlenk operator for some ordinal $\alpha$.

\nr{ii} $T$ is an Asplund operator.
\end{proposition}

\begin{proof}
First suppose that $T$ is Asplund. By the {\reinov}-Heinrich-Stegall factorization theorem for Asplund operators (c.f. Section~\ref{greatfamily}), there exists an Asplund space $G$ such that $I_G$ factors $T$. By Proposition~\ref{collection}(iii), there is an ordinal $\alpha$ such that $\mszlenk{G}=\omega^\alpha$, hence $\mszlenk{T}\leqslant \mszlenk{I_G} = \mszlenk{G} = \omega^\alpha$. That is, $T$ is $\alpha$-Szlenk.

Now suppose that $T$ is not Asplund. By Proposition~\ref{chickennuggets}, there exists a separable Banach space $D$ and an operator $S:D\To E$ such that $S\dual T\dual \cball{F\dual}$ is nonseparable in the norm topology. As $D$ is norm separable, we have that $S\dual T\dual \cball{F\dual}$ is $w\dual$-metrizable, hence by Lemma~\ref{nickenchuggets} it follows that $\mszlenk{TS} = \mszlenk{S\dual T\dual \cball{F\dual}} = \infty$. That is, $TS$ fails to be $\alpha$-Szlenk for any ordinal $\alpha$. As the classes $\szlenkop{\alpha}$ are operator ideals, $T$ fails to be $\alpha$-Szlenk for any $\alpha$.
\end{proof}

For every pair of Banach spaces $(E, \, F)$, there is an ordinal $\alpha$ such that if $T\in \allop (E, \, F)$ is $\beta$-Szlenk for some ordinal $\beta$, then $T$ is $\alpha$-Szlenk. Indeed, we may take $\alpha$ to satisfy $\omega^\alpha = \sup \set{\mszlenk{T} \mid T\in \szlenkop{\beta} (E, \, F) \mbox{ for some ordinal }\beta}$. By Proposition~\ref{opequiv}, with $\alpha$ so defined we have $\asplundop (E, \, F) = \szlenkop{\alpha}(E, \, F)$.

We now determine the relationship between the operator ideals $\szlenkop{\alpha}$ and the operator ideal $\sepop$ of operators having separable range. In what follows, $\sepop\dual$ denotes the operator ideal of operators $T$ with $T\dual\in\sepop$. The following result is essentially an operator-theoretic generalization of Proposition~\ref{collection}(iv).

\begin{proposition}\label{sepsepseppy}
\begin{equation*} \displaystyle \sepop\dual = \sepop \cap \asplundop =  \sepop\cap\bigcup_{\alpha \in\ord} \szlenkop{\alpha} = \sepop \cap \bigcup_{\alpha < \omega_1}\szlenkop{\alpha} = \sepop \cap \szlenkop{\omega_1}.\end{equation*}
\end{proposition}

\begin{proof}
First note that $\szlenkop{\omega_1}\subseteq \bigcup_{\alpha<\omega_1}\szlenkop{\alpha}$. Indeed, for $T\in \szlenkop{\omega_1}$ we have $cf (\mszlenk{T}) = cf(\sup \{  \meeszlenk{1/n}{T}\mid n\in\nat\}) \leqslant \omega$, whereas $cf (\omega^{\omega_1}) = \omega_1$. We thus deduce that $\mszlenk{T} <\omega_1$, hence $T\in \bigcup_{\alpha<\omega_1}\szlenkop{\alpha}$. By this observation and Proposition~\ref{opequiv} we have
\begin{equation*}
\sepop \cap \szlenkop{\omega_1} = \sepop \cap \bigcup_{\alpha < \omega_1}\szlenkop{\alpha} \subseteq \sepop \cap \bigcup_{\alpha \in\ord} \szlenkop{\alpha} = \sepop \cap \asplundop\, ,
\end{equation*}
and so it now suffices to show that $\sepop \cap \asplundop \subseteq \sepop\dual$ and $\sepop\dual \subseteq \sepop \cap \szlenkop{\omega_1}$.

To prove $\sepop \cap \asplundop \subseteq \sepop\dual$, we first note that Heinrich \cite{Heinrich1980} has shown that $(\asplundop, \, \asplundop)$ is a $\Sigma_p$-pair for every $1<p<\infty$. We \emph{claim} that $(\sepop, \, \sepop)$ is also a $\Sigma_p$-pair for all $1<p<\infty$. To verify our claim, we note that if $(E_m)_{m\in \nat}$ and $(F_n)_{n\in\nat}$ are sequences of Banach spaces, $1<p<\infty$ and $T\in \allop ((\bigoplus_{m \in \nat}E_m)_p , \, (\bigoplus_{n\in\nat}F_n)_p )$ is such that $T\notin \sepop$, then the set $\bigcup\{ Q_\gee TU_\eff (\textstyle \bigoplus_{m \in \nat}E_m)_p \mid \eff,\, \gee \in \nat\fin\} $ is nonseparable since its uniform closure contains $T(\bigoplus_{m \in \nat}E_m)_p$. As $\nat\fin$ is countable, it follows that are $\eff, \, \gee\in\nat\fin$ such that $Q_\gee TU_\eff ( \bigoplus_{m \in \nat}E_m)_p$ is nonseparable. That is, $Q_\gee TU_\eff\notin\sepop$. This completes the proof of the claim, and it follows that $(\sepop\cap\asplundop, \, \sepop\cap\asplundop)$ is a $\Sigma_p$-pair for all $1<p<\infty$. Moreover, $\sepop\cap\asplundop$ is injective and surjective since the same is true for $\sepop$ and $\asplundop$. Thus, by Theorem~\ref{manna}, every element of $\sepop\cap\asplundop$ factors through a separable Asplund space. By Theorem~\ref{Aspequiv}, this implies that every element of $\sepop\cap\asplundop$ factors through a Banach space with separable dual, and the inclusion $\sepop \cap \asplundop \subseteq \sepop\dual$ follows.

We now show that $\sepop\dual \subseteq \sepop \cap \szlenkop{\omega_1}$. The inclusion $\sepop\dual \subseteq \sepop$ is well-known (see, for example, \cite[Proposition~4.4.8]{Pietsch1980}),  so we need only show that $\sepop\dual \subseteq \szlenkop{\omega_1}$. To this end, note that similar arguments to those used above show that $(\sepop\dual, \, \sepop\dual)$ is a $\Sigma_p$-pair for every $1<p<\infty$. Moreover, $\sepop\dual$ is injective and surjective, hence Theorem~\ref{manna} implies that every element of $\sepop\dual$ factors through a Banach space with separable dual. By Proposition~\ref{collection}(iv), this means that every element of $\sepop\dual$ factors through a Banach space of countable Szlenk index; the inclusion $\sepop\dual \subseteq \szlenkop{\omega_1}$ follows.
\end{proof}

To conclude the current section we mention two sequential variants of the Szlenk index that have appeared in the literature. Sequential definitions are often advantageous from a utilitarian point-of-view, but - as we shall now see - they do not seem to be sufficient for the development of a general theory of operator ideals associated with the Szlenk index such as that that initiated here.

For $E$ a Banach space, $K\subseteq E\dual$ and $\eps>0$, we define derivations
\begin{equation*}
m_\eps (K) :=\{x\dual \in K \mid \exists (x_n\dual)_n \subseteq K, \, x_n\dual \stackrel{w\dual}{\rightarrow} x\dual , \, \Vert x_n\dual - x\dual\Vert \geqslant \eps \mbox{ for all }n\in \nat \}
\end{equation*}
and
\begin{align*} n_\eps(K):= \{ x\dual \in K  \mid &\exists (x_n\dual) \subseteq K, \, \exists (x_n) \subseteq \cball{E}, x_n\dual \stackrel{w\dual}{\rightarrow} x\dual , \, x_n \stackrel{w}{\rightarrow} 0, \, \overline{\lim}_n \abs{\functional{x_n\dual}{x_n}}\geqslant \eps \}
\end{align*}
on $K$. As with $s_\eps$, we may iterate $m_\eps$ and $n_\eps$ to obtain derivations $m_\eps^\alpha$ and $n_\eps^\alpha$ for $\eps > 0$ and $\alpha$ an ordinal, with corresponding indices ${\rm Mz}_\eps(K)$, ${\rm Mz}(K)$, ${\rm Nz}_\eps(K)$ and ${\rm Nz}(K)$. Analogously to the definition of the classes $\szlenkop{\alpha}$, for each ordinal $\alpha$ we define $\mzlenkop{\alpha}:= \set{T\in\allop \mid {\rm Mz}(T)\leqslant \omega^\alpha}$ and $\nzlenkop{\alpha}:= \set{T\in \allop\mid {\rm Nz}(T) \leqslant \omega^\alpha} $.

The main obstacle to proving that the classes $\mzlenkop{\alpha}$ form operator ideals is that we do not seem to have an analogue of Lemma~\ref{poopoopoo} for the derivations $m_\eps^\alpha$ (since it need not be the case that every sequence in $K$ has a $w\dual$-convergent sub\emph{sequence}). However, we may form operator ideals from the classes $\mzlenkop{\alpha}$ by taking their intersection with the class $\wscop$ consisting of all operators having $w\dual$-sequentially compact adjoint. That is, an operator $T:E\To F$ belongs to $\wscop$ if and only if $T\dual\cball{F\dual}$ is $w\dual$-sequentially compact. Standard arguments, similar to those used to show the same for $\compactop$, show that $\wscop$ is a closed, injective, surjective operator ideal. A proof similar to that of Theorem~\ref{idealthm} shows that $\mzlenkop{\alpha}\cap \wscop$ is a closed, injective, surjective operator ideal for every ordinal $\alpha$.
Moreover, it is elementary to show that the indices ${\rm Sz}$ and ${\rm Mz}$ coincide for operators $T:E\To F$ with the property that $T\dual \cball{F\dual}$ is $w\dual$-metrizable; this is the case precisely when the range of $T$ is norm separable. Thus, when dealing with operators having separable range, one may usually work with ${\rm Mz}$ in place of ${\rm Sz}$ if it is more convenient.

We now discuss the index ${\rm Nz}$ and the associated classes $\nzlenkop{\alpha}$. The index ${\rm Nz}$ is in fact that introduced by Szlenk in \cite{Szlenk1968}, and it coincides with ${\rm Sz}$ and ${\rm Mz}$ for operators whose domain is a separable Banach space containing no subspace isomorphic to $\ell_1$ \cite[Proposition~3.3]{Lancien1993}. However, the index ${\rm Nz}$ lacks sufficiently good permanence properties for the classes $\nzlenkop{\alpha}$ to be operator ideals over the class of \emph{all} Banach spaces. We illustrate this claim by way of the following simple example. Let $U:\ell_2\To \ell_\infty$ be an isometric linear embedding. As observed by J.~Bourgain \cite[p.88]{Bourgain1979}, if $E$ is a Grothendieck space with the Dunford-Pettis property, then ${\rm Nz}(E)= 1$. In particular, ${\rm Nz}(\ell_\infty)=1$. We thus have $I_{\ell_\infty} \in \nzlenkop{0}$. On the other hand, $I_{\ell_\infty}U \notin \nzlenkop{0}$ since ${\rm Nz}(I_{\ell_\infty}U) = {\rm Nz}(\cball{\ell_2\dual}) = \omega >\omega^0$. In particular, $\nzlenkop{0}$ fails to satisfy  condition ${\rm OI}_3$ of Definition~\ref{idealdefn}. Similar examples, based on the spaces defined by the construction of Szlenk (see Example~\ref{lastofthird} of the current paper), show that $\nzlenkop{\alpha+1}$ fails to satisfy ${\rm OI}_3$ for every $\alpha<\omega_1$.

Despite the apparent deficiency of the index ${\rm Nz}$ from the point of view of developing a theory of operator ideals associated with the Szlenk index, we wish to emphasize the importance of the index ${\rm Nz}$ in the study of the structure of operators acting on spaces $C(K)$, where $K$ is a metrizable compact space. Indeed, a number of authors have studied the connections between the ${\rm Nz}$ index of operators acting on $C(K)$ spaces and `fixing' properties of such operators; we refer to \cite{Rosenthal2003} for a survey, and to the work of I.~Gasparis \cite{Gasparis2005} for more recent results. In fact, we believe that both of the indices ${\rm Sz}$ and ${\rm Nz}$ are of interest in the study of operators in $\allop (C[0,\, 1])$. For example, the following question is of interest in studying the closed ideal structure of the Banach algebra $\allop (C[0, \, 1])$:
\begin{question}\label{handingin}
Let $R\in \sepop\dual(C[0, \, 1])$. Does there exist $S\in \wcompactop (C[0, \, 1])$ such that ${\rm Sz}(R+S)={\rm Nz}(R+S)$?
\end{question} Question~\ref{handingin} asks whether the indices ${\rm Sz}$ and ${\rm Nz}$ coincide on $\sepop\dual(C[0, \, 1])$ up to a weakly compact perturbation. The motivation for Question~\ref{handingin} is the fact that $\wcompactop$ is a closed operator ideal and that, for $T\in \allop (C[0, \, 1])$, ${\rm Nz}(T)$ is minimal (that is, is equal to $1$) if and only if $T$ is weakly compact; this latter fact regarding minimality of ${\rm Nz}(T)$ for $T\in \allop (C[0, \, 1])$ is due to D.~Alspach \cite[Remark~2]{Alspach1978}.

\section{Examples}\label{yemem}
In this section we discuss the algebras $\szlenkop{\alpha}(E)$ for a number of well-known Banach spaces $E$. In particular, we study the place of the ideals $\szlenkop{\alpha}(E)$ in the lattice of closed, two-sided ideals of $\allop (E)$ by relating them to other well-known closed ideals (for example, the ideal of weakly compact operators).

\begin{example}
Our first example is the Banach space $L_\infty = L_\infty [0, \, 1]$. We will show that the operator ideals $\wcompactop$, $\overline{\fac{\ell_2}}$, $\asplundop$, $\szlenkop{1}$ and $\sepop\dual$ coincide on $L_\infty$. For this purpose, we state the following impressive result of H.~Jarchow \cite{Jarchow1986}:
\begin{theorem}\label{yarch}
Let $A$ be a $C\dual$-algebra and $F$ a Banach space. Then \begin{equation*} \wcompactop (A, \, F) = \overline{\Gamma_2}\inj (A, \, F) .\end{equation*}
\end{theorem}

We also require the following lemma.
\begin{lemma}\label{sepstep}
Let $\mathcal{H}$ be a Hilbert space and $E$ a Banach space such that every weakly compact subset of $E$ is norm separable. Then $\allop (\mathcal{H}, \, E) = \fac{\ell_2}(\mathcal{H}, \, E)$.
\end{lemma}
\begin{proof}
Let $T\in \allop (\mathcal{H}, \, E)$. The reflexivity of $\mathcal{H}$ implies $T\in \wcompactop (\mathcal{H}, \, E)$, and the norm separability of $T(\cball{\mathcal{H}})$ implies $T\in \sepop$. Thus, by the (separable) DFJP factorization theorem \cite[Lemma~1(xi)]{Davis1974}, there is a separable, reflexive Banach space $F$ and operators $A: \mathcal{H} \longrightarrow F$ and $B: F\longrightarrow E$ such that $T=BA$. Since $F\dual$ is separable, we have that $\overline{A\dual(F\dual)}$ is isometric to a separable closed subspace of $\mathcal{H}\dual$, hence isometric to a closed subspace of $\ell_2$. Thus $A\dual \in \fac{\ell_2}$. Making the identifications $\mathcal{H} = \mathcal{H}^{\ast \ast}$ and $F = F^{\ast \ast}$ via the canonical injections of $\mathcal{H}$ and $F$ into their second duals, we have $A = A^{\ast \ast} \in \fac{\ell_2\dual} = \fac{\ell_2}$, hence $T = BA \in \fac{\ell_2}$.
\end{proof}

The inclusion $\sepop (L_\infty) \subseteq \wcompactop (L_\infty)$ holds since $L_\infty$ is a Grothendieck space, and $\wcompactop (L_\infty) \subseteq \sepop (L_\infty)$ since weakly compact subsets of $L_\infty$ are norm separable \cite[Proposition~4.7]{Rosenthal1970a}. Thus $\sepop (L_\infty) = \wcompactop (L_\infty)$. As $\mszlenk{\mathcal{H}} = \omega$ for $\mathcal{H}$ a Hilbert space (see \cite[p.106]{Odell2004}), and since $L_\infty$ is both a $C\dual$-algebra and an injective Banach space, Theorem~\ref{yarch} yields $\sepop (L_\infty) = \wcompactop (L_\infty)  = \overline{\Gamma_2} (L_\infty) \subseteq \szlenkop{1} (L_\infty) \subseteq \asplundop  (L_\infty)$. We have $\sepop\dual = \sepop \cap \asplundop$ by Proposition~\ref{sepsepseppy}, hence $\sepop \dual (L_\infty) = \sepop (L_\infty)$. Thus, to show that $\wcompactop$, $\cfac{\ell_2}$, $\asplundop$, $\szlenkop{1}$ and $\sepop\dual$ coincide on $L_\infty$, it now suffices to show that $\asplundop (L_\infty) \subseteq \wcompactop (L_\infty)$ and $\Gamma_2 (L_\infty) \subseteq \fac{\ell_2}(L_\infty)$. The first of these inclusions is justified by the fact that nonweakly compact operators on $L_\infty$ fix a copy of the non-Asplund space $L_\infty$ (see \cite[Proposition~1.2]{Rosenthal1970} and the main result of \cite{Pelczy'nski1958}). The second inclusion follows from Lemma~\ref{sepstep} and the fact that every weakly compact subset of $L_\infty$ is norm separable.
\end{example}

\begin{example}
For our next example, we consider the space $L_1 = L_1[0, \, 1]$. Similarly to the previous example, we will show that the operator ideals $\wcompactop$, $\overline{\fac{\ell_2}}$, $\asplundop$, $\szlenkop{1}$ and $\sepop\dual$ coincide on $L_1$.

Let $Q : L_1\hookrightarrow L_1^{\ast \ast}$ denote the canonical embedding and let $P : L_1^{\ast \ast} \longrightarrow L_1$ be a projection (that $L_1$ is complemented in its bidual is well-known; see, for example, \cite[Proposition~6.3.10]{Albiac2006}). It is clear that, since $\mszlenk{\ell_2}=\omega$, we have $\cfac{\ell_2}(L_1) \subseteq \szlenkop{1} (L_1) \subseteq \asplundop (L_1)$. Moreover, since $\sepop\dual$ and $\asplundop$ coincide on separable Banach spaces by Proposition~\ref{sepsepseppy}, it suffices to show that $\asplundop (L_1) \subseteq \wcompactop (L_1)$ and $\wcompactop (L_1) \subseteq \cfac{\ell_2}(L_1)$. The first of these inclusions is justified by the fact that nonweakly compact operators into $L_1$ fix a copy of $\ell_1$ \cite[Theorem~1]{Pelczy'nski1965a}, and therefore fail to be Asplund. For the second inclusion, let $T \in \wcompactop (L_1)$. Then, by Gantmacher's theorem, $T\dual$ is a weakly compact operator on the (up to isomorphism) $C\dual$-algebra $L_1\dual$. Moreover, $L_1\dual$ is an injective Banach space, hence Theorem~\ref{yarch} ensures the existence of a sequence $(S_n)$ in $\Gamma_2 (L_1\dual)$ satisfying $\norm{T\dual - S_n}\rightarrow 0$. It follows then that, since $T = PT^{\ast \ast}Q$, we have $\norm{T - PS_n\dual Q} = \norm{P(T^{\ast \ast} - S_n\dual)Q} \rightarrow 0 $. In particular, $T \in \overline{\Gamma}_2 (L_1)$ since $S_n\dual \in \Gamma_2$ for all $n$. As $L_1$ is separable, it follows that $T\in \cfac{\ell_2}(L_1)$.
\end{example}

\begin{example}\label{drivingeg}
We now consider the ideals $\szlenkop{\alpha}(C[0, \, 1])$. The lattice of closed, two-sided ideals in $\allop (C[0,\, 1])$ contains the following linearly ordered chain, where  $0<\beta <\omega_1$:
\begin{align*}
\set{0} \subsetneq \compactop (C[0, \, &1]) = \szlenkop{0}(C[0, \, 1]) \subsetneq \wcompactop (C[0, \, 1]) \subsetneq \szlenkop{1}(C[0, \, 1]) \subseteq \ldots \\ 
\ldots \subseteq & \overline{\bigcup_{\gamma<\beta} \szlenkop{\gamma}(C[0, \, 1])} \subseteq\szlenkop{\beta}(C[0, \, 1]) \subsetneq \szlenkop{\beta +1}(C[0, \, 1]) \subseteq \ldots \\
&\ldots \bigcup_{\alpha <\omega_1}\szlenkop{\alpha}(C[0, \, 1]) = \asplundop (C[0, \, 1]) = \sepop\dual(C[0, \, 1]) \subsetneq \allop (C[0, \, 1]).
\end{align*}
Note that the ideal $\sepop\dual(C[0, \, 1])$ is the unique maximal ideal in $\allop(C[0, \, 1])$ since each element of $\allop(C[0, \, 1]) \setminus \sepop\dual (C[0, \, 1])$ factors the identity operator of $C[0, \, 1]$. Indeed, combining theorems of H.~Rosenthal \cite[Theorem~1]{Rosenthal1972} and A.~{\pelch} \cite[Theorem~1]{Pelczy'nski1968}, for any $T\in \allop(C[0, \, 1]) \setminus \sepop\dual (C[0, \, 1])$ there exists a closed subspace $E \subseteq C[0, \, 1]$ such that $T|_E$ is an isomorphism, $E$ is isomorphic to $C[0, \, 1]$ and $T(E)$ is complemented in $C[0, \, 1]$. Let $R$ be an isomorphism of $C[0, \, 1]$ onto $E$, let $P:C[0, \,1]\To T(E)$ be a continuous projection and set $V=(TR)^{-1}P$. Then $I_{C[0, \, 1]} = VTR$.

We now justify the other claims above regarding the lattice of closed ideals in $\allop (C[0,\, 1])$. With $A: C[0,\, 1]\To \ell_2$ a surjective operator and $B:\ell_2\To C[0, \, 1]$ noncompact, $BA \in \wcompactop (C[0, \, 1]) \setminus \compactop (C[0,\, 1])$. That $\wcompactop (C[0,\, 1]) \subseteq \szlenkop{1}(C[0,\, 1])$ follows from Theorem~\ref{yarch} and the fact that, since Hilbert spaces have Szlenk index $\omega$ and $\szlenkop{1}$ is closed and injective, $\overline{\Gamma_2}\inj \subseteq \szlenkop{1}$. Any projection of $C[0, \, 1]$ onto a subspace isomorphic to $c_0$ (of which there are many) belongs to the difference $\szlenkop{1}(C[0,\, 1]) \setminus \wcompactop (C[0, \, 1])$ since $c_0$ is nonreflexive and of Szlenk index $\omega$. Similarly, the difference $\szlenkop{\beta +1}(C[0, \, 1]) \setminus \szlenkop{\beta}(C[0, \, 1])$ contains any projection of $C[0, \, 1]$ onto a subspace isomorphic to $C(\omega^{\omega^\beta}+1)$ (here, $\omega^{\omega^\beta}+1$ is equipped with its (compact) order topology; see \cite{H'ajek2007} for a proof that $\mszlenk{C(\omega^{\omega^\beta}+1)} = \omega^{{\beta+1}}$ for $\beta<\omega_1$). That the operator ideals $\bigcup_{\alpha <\omega_1}\szlenkop{\alpha}$, $\asplundop$ and $ \sepop\dual$ coincide on $C[0, \, 1]$ follows from Proposition~\ref{sepsepseppy}.
\end{example}

\begin{example}\label{lastofthird}
Let $V$ denote the complementably universal unconditional basis space of {\pelch} \cite{Pelczy'nski1969}. Then, as in the case of $C[0, \, 1]$ above, we have $\szlenkop{\beta}(V) \subsetneq \szlenkop{\beta+1}(V)$ for every $\beta <\omega_1$. To show this, it suffices to find, for each $\beta<\omega_1$, a Banach space $G_\beta$ having an unconditional basis and Szlenk index $\omega^{\beta+1}$. Indeed, the existence of such a space ensures the existence of a projection of $V$ onto a complemented subspace isomorphic to $G_\beta$, and such a projection clearly belongs to $\szlenkop{\beta+1}(V)\setminus \szlenkop{\beta}(V)$. For the existence of the desired spaces $G_\beta$, we turn to Szlenk's construction in \cite{Szlenk1968} of a family of separable, reflexive Banach spaces whose Szlenk indices are (collectively) unbounded above in $\omega_1$. The construction is as follows: Let $E_0 = \set{0}$, $E_{\alpha +1} = E_\alpha \oplus_1 \ell_2 $  for $\alpha <\omega_1$ and, if $\alpha<\omega_1$ is a limit ordinal, $E_\alpha = (\bigoplus_{\gamma <\alpha}E_\gamma )_2$. A straightforward transfinite induction on $\alpha <\omega_1$ shows that $E_\alpha$ has a $1$-unconditional basis for all nonzero $\alpha<\omega_1$. Moreover, a slight modification of arguments in \cite[proof of Proposition~2.16]{Brookera} show that for each $\beta <\omega_1$ there exists $\alpha(\beta)<\omega_1$ such that $\mszlenk{E_{\alpha(\beta)}}=\omega^{\beta+1}$. Taking $G_\beta = E_{\alpha(\beta)}$ gives the desired spaces $G_\beta$ ($\beta<\omega_1$).

Finally, note that $\bigcup_{\alpha <\omega_1}\szlenkop{\alpha}(V) = \asplundop (V) = \sepop\dual(V) \subsetneq \allop (V)$ by Proposition~\ref{sepsepseppy} and the existence of $P\in \allop (V)\setminus \sepop\dual(V)$ with $P^2 = P$ and $P(V)$ isomorphic to $\ell_1$.
\end{example}

\section{Quantitative factorization of Asplund operators}\label{feekwhack}

An important, basic question in operator ideal theory is whether a given operator ideal $\opideal$ has the factorization property; that is, whether every element of $\opideal$ factors through a Banach space whose identity operator belongs to $\opideal$.
The most well-known and widely applied result in this direction is the celebrated Davis-Figiel-Johnson-{\pelch} factorization theorem \cite{Davis1974} asserting that every weakly compact operator factors through a reflexive Banach space. In the absence of the factorization property, one may then ask whether $\opideal$ satisfies some nontrivial `weak' factorization property. For example, W.~Johnson has shown in \cite{Johnson1971} that there exists a separable, reflexive Banach space $E$ with the property that every approximable operator (= uniform limit of finite-rank operators) factors through $E$ with approximable factors. In this and subsequent sections of the current paper we study factorization properties of the operator ideals $\szlenkop{\alpha}$.

Our main task in this section is to establish the following weak factorization result for the operator ideals $\szlenkop{\alpha}$. In light of Proposition~\ref{opequiv} and Proposition~\ref{collection}(ii), this result can be considered a quantitative refinement of the independent efforts of {\reinov}, Heinrich and Stegall (c.f. Section~\ref{greatfamily}) showing that the operator ideal of Asplund operators possesses the factorization property.

\begin{theorem}\label{weakfac}
For $\alpha$ an ordinal, $\szlenkop{\alpha}$ has the $\szlenkop{\alpha+1}$-factorization property. That is, each $T\in \szlenkop{\alpha}$ factors through a Banach space whose Szlenk index is at most $\omega^{\alpha+1}$.
\end{theorem}

Before embarking on a proof of Theorem~\ref{weakfac}, we mention a similar result due to B.~Bossard. It is shown in \cite[Theorem~3.9]{Bossard1997} that there is a universal function $\varphi:\omega_1 \To \omega_1$ such that for any separable Banach spaces $E$ and $F$ and any Asplund operator $T: E\To F$, there exists a Banach space $G$ and operators $A: E\To G$ and $B: G\To F$ such that $G$ has a shrinking basis,  $\mszlenk{G}\leqslant \varphi (\mszlenk{T})$ and $T=BA$. It will be shown at the end of Section~\ref{theredethere} that $\varphi$ necessarily exceeds the identity function of $\omega_1$ at uncountably many points of $\omega_1$.

We shall deduce Theorem~\ref{weakfac} from the following proposition.

\begin{proposition}\label{unifac}
Let $\Lambda$ and $\Upsilon$ be sets, $\set{E_\lambda \mid \lambda \in \Lambda}$ and $\set{F_\upsilon \mid \upsilon \in \Upsilon}$ families of Banach spaces, $p=0$ or $1<p<\infty$, $T: (\bigoplus_{\lambda \in \Lambda}E_\lambda)_p \To (\bigoplus_{\upsilon \in \Upsilon}F_\upsilon)_p$ an operator and $\alpha >0$ an ordinal. The following are equivalent:

\nr{i} $\mszlenk{T}\leqslant \omega^\alpha$.

\nr{ii} $\sup \set{\meeszlenk{\eps}{TU_\eff}\mid \eff\in\Lambda\fin}<\omega^\alpha $ for every $\eps >0$.

\nr{iii} $\sup \set{\meeszlenk{\eps}{Q_\gee T}\mid \gee\in\Upsilon\fin}<\omega^\alpha $ for every $\eps >0$.

\nr{iv} $\sup \set{\meeszlenk{\eps}{Q_\gee TU_\eff}\mid \eff\in\Lambda\fin, \, \gee\in\Upsilon\fin}<\omega^\alpha $ for every $\eps >0$.
\end{proposition}

Let us now see how Theorem~\ref{weakfac} follows from Proposition~\ref{unifac}. We begin by letting $1<p<\infty$. By Theorem~\ref{idealthm} and Theorem~\ref{manna}, it suffices to show that $(\szlenkop{\alpha}, \, \szlenkop{\alpha+1})$ is a $\Sigma_p$-pair. To this end, let $(E_m)_m$ and $(F_n)_n$ be sequences of Banach spaces and suppose $T\in \allop ((\bigoplus_{m\in \nat}E_m)_p, \,(\bigoplus_{n\in\nat}F_n)_p )$ is such that $Q_\gee T U_\eff \in \szlenkop{\alpha}$ for all $\eff, \, \gee \in \nat\fin$. Then
\begin{equation*}
\forall \eps>0 \quad \sup \set{{\rm Sz}_\eps (Q_\gee TU_\eff)\mid \eff, \, \gee \in \nat\fin}\leqslant \omega^{\alpha}<\omega^{\alpha+1},
\end{equation*} hence $T\in \szlenkop{\alpha+1}$ by Proposition~\ref{unifac}, and we are done.

To prove Proposition~\ref{unifac}, we draw on several preliminary results. The first of these is the following variant of \cite[Proposition~2.2]{H'ajek2007}, which can be useful for obtaining an upper estimate on the Szlenk index of an operator.

\begin{proposition}\label{dadeeeee}
Let $E$ and $F$ be Banach spaces, $T:E\To F$ an operator and $\beta$ an ordinal. Suppose that for every $\eps > 0$ there exists $\beta_\eps <\omega^\beta$ and $\delta_\eps \in (0, \, 1)$ such that $s^{\beta_\eps}_{\eps} (T\dual \cball{F\dual} ) \subseteq \delta_\eps T\dual \cball{F\dual}$. Then $\mszlenk{T} \leqslant \omega^\beta$.
\end{proposition}

\begin{proof}
Fix $\eps>0$. We \emph{claim} that $\mdset{{\beta_\eps}\cdot n}{\eps}{T\dual \cball{F\dual}} \subseteq \delta_\eps^nT\dual \cball{F\dual} $ for all $n \in \nat$. Indeed, it is true for $n=1$ by assumption, and if true for $n\leqslant k$ then
\begin{equation*}
\mdset{\beta_\eps \cdot (k+1)}{\eps}{T\dual \cball{F\dual}} \subseteq \mdset{\beta_\eps}{\eps}{\delta_\eps^k T\dual \cball{F\dual}} = \delta_\eps^k \mdset{\beta_\eps}{\eps / \delta_\eps^k}{T\dual \cball{F\dual}} \subseteq \delta_\eps^k \mdset{\beta_\eps}{\eps}{T\dual \cball{F\dual}} \subseteq \delta_{\eps}^{k+1}T\dual \cball{F\dual}, 
\end{equation*} 
so that the above claim holds by induction on $n$. 

For each $\eps>0$ let $N_\eps \in \nat$ be large enough that $\mdset{{\beta_\eps}\cdot N_\eps}{\eps}{T\dual \cball{F\dual}} \subseteq \frac{\eps}{2} \cball{E\dual}$. Then $\mdset{{\beta_\eps}\cdot N_\eps +1}{\eps}{T\dual \cball{F\dual}} = \emptyset$ for each $\eps>0$, hence $ \mszlenk{T} \leqslant \sup_{\eps >0}\, ({\beta_\eps}\cdot N_\eps +1) \leqslant \omega^\beta$.
\end{proof}

The next two lemmas concern the action of the $\eps$-Szlenk derivations on $w\dual$-compact sets contained in the dual of a direct sum of Banach spaces. The first is a discrete variant of \cite[Lemma~3.3]{H'ajek2007} and is proved in \cite[Lemma~2.6]{Brookera}.

\begin{lemma}\label{tvl}
Let $\Lambda$ be a set, $\set{E_\lambda \mid \lambda \in \Lambda}$ a family of Banach spaces, $1\leqslant q<\infty$, $p$ predual to $q$ and $K \subseteq (\bigoplus_{\lambda \in \Lambda}E_\lambda)_{p}\dual $ a nonempty $w\dual$-compact set. Let $\alpha$ be an ordinal, $\arr \subseteq \Lambda$ and $\eps >\delta >0$. If $x \in \mdset{\alpha}{\eps}{K}$ is such that $\norm{U_\arr\dual \, x}^{q} > \abs{K}^q-(\frac{\eps -\delta}{2})^q$, then $U_\arr\dual\, x \in \mdset{\alpha}{\delta}{U_\arr\dual\, K}$.
\end{lemma}

\begin{lemma}\label{baconlemma}
Let $\Upsilon$ be a set, $\set{F_\upsilon \mid \upsilon \in \Upsilon}$ a family of Banach spaces, $E$ a Banach space, $1\leqslant q<\infty$, $p$ predual to $q$, $K \subseteq (\bigoplus_{\lambda \in \Lambda}E_\lambda)_{p}\dual $ a nonempty $w\dual$-compact set, $T: E \To \left( \bigoplus_{\upsilon \in \Upsilon}F_\upsilon \right)_p $ a nonzero operator and $\eps >0$. Let $\alpha$ be an ordinal and let $x \in s^\alpha_\eps (T\dual K )$. Then there is $y\in s^\alpha_{\eps /(2\norm{T})}(K)$ such that $ T\dual y = x$. Further, if $\ess \subseteq \Upsilon$ is such that $ \norm{Q_\ess\dual V_\ess\dual y}^{q} > \abs{K}^q-(\eps /( 8\norm{T}))^q$, then $T\dual Q_\ess\dual V_\ess\dual y \in \mdset{\alpha}{\eps / 4}{T\dual Q_\ess\dual V_\ess\dual K}$.
\end{lemma}

\begin{proof}
The lemma is clearly true for $\alpha = 0$. We now assume it true for some $\alpha = \gamma$, and show that it is also true for $\alpha = \gamma +1$. To this end, let $x\in s^{\gamma+1}_\eps (T\dual K) = s_\eps (s^\gamma_\eps (T\dual K))$. Then there exists a net $(x_i)_{i\in I}$ in $s^\gamma_\eps (T\dual K)$ such that $x_i \stackrel{w\dual}{\rightarrow}x$ and $\norm{x_i - x}>\eps /2$ for all $i \in I$ (for example, take $I$ to be the set of all $w\dual$-neighbourhoods of $x$, ordered by reverse set inclusion). For each $i\in I$ let $y_i \in s^\gamma_{\eps /(2\norm{T})}(K)$ be such that $T\dual y_i = x_i$. Passing to a subnet, we may assume that $(y_i)_{i\in I}$ has a $w\dual$-limit $y \in s^\gamma_{\eps /(2\norm{T})}(K)$, and then $T\dual y = x$. Moreover, as $\norm{y_i -y}\geqslant \norm{x_i -x}/ \norm{T}>\eps /(2\norm{T})$ for all $i$, we have $y \in s^{\gamma+1}_{\eps /(2\norm{T})}(K)$. Now suppose that $\norm{Q_\ess\dual V_\ess\dual y}^q > \abs{K}^q -(\eps /(8\norm{T}))^q$, where $\ess \subseteq \Upsilon$. Passing to a subnet, we may assume $\norm{Q_\ess\dual V_\ess\dual y_i}^q > \abs{K}^q -(\eps /(8\norm{T}))^q$ for all $i$. Then for all $i$, \begin{equation*} \norm{y_i - Q_\ess\dual V_\ess\dual y_i}= (\norm{y_i}^q - \norm{Q_\ess\dual V_\ess\dual y_i}^q)^{1/q} \leqslant (\abs{K}^q - \norm{Q_\ess\dual V_\ess\dual y_i}^q)^{1/q} < \frac{\eps}{ 8\norm{T}}\,,\end{equation*} hence $\norm{y - Q_\ess\dual V_\ess\dual y} \leqslant \eps /(8 \norm{T})$ by $w\dual$-lower semicontinuity. Thus, for all $i$,
\begin{align*}
\Vert T\dual Q_\ess\dual V_\ess\dual y_i - T\dual &Q_\ess\dual V_\ess\dual y \Vert \\ &\geqslant \norm{T\dual y_i - T\dual y} - \norm{T\dual y - T\dual Q_\ess\dual V_\ess\dual y} - \norm{T\dual y_i - T\dual Q_\ess\dual V_\ess\dual y_i} \\ &\geqslant \norm{x_i -x} - \norm{T}\norm{y - Q_\ess\dual V_\ess\dual y} - \norm{T}\norm{y_i - Q_\ess\dual V_\ess\dual y_i} \\ &> \frac{\eps}{2} - 2\norm{T} \cdot \frac{\eps}{ 8\norm{T}} =\frac{\eps}{4} \, .
\end{align*}
Since $T\dual Q_\ess\dual V_\ess\dual y_i \in s^\gamma_{\eps /4}(T\dual Q_\ess\dual V_\ess\dual K)$ for each $i$ by the induction hypothesis, and since $T\dual Q_\ess\dual V_\ess\dual y_i \stackrel{w\dual}{\rightarrow} T\dual Q_\ess\dual V_\ess\dual y$, we have $T\dual Q_\ess\dual V_\ess\dual y \in s^{\gamma +1}_{\eps /4}(T\dual Q_\ess\dual V_\ess\dual K)$. Thus, in particular, the assertion of the lemma passes to successor ordinals.

Finally, let $\gamma$ be a limit ordinal and suppose that the assertion of the lemma holds whenever $\alpha <\gamma$. Let $x\in s^\gamma_\eps (T\dual K) = \bigcap_{\alpha <\gamma}s^\alpha_\eps (T\dual K)$ and for each $\alpha <\gamma $ let $y_\alpha \in s^\alpha_{\eps /(2\norm{T})}(K)$ be such that $T\dual y_\alpha = x$. By $w\dual$-compactness, there is a directed set $J$ and a mapping $f: J\To \gamma$ such that $(y_{f(j)})_{j\in J}$ is a $w\dual$-convergent subnet of $(y_\alpha)_{\alpha<\gamma}$. Let $y$ denote the $w\dual$-limit of $(y_{f(j)})_{j\in J}$. Then $T\dual y = x$, and since $f(J)$ is cofinal in $\gamma$ (by definition of a subnet), $y \in \bigcap_{j\in J}s^{f(j)}_{\eps /(2\norm{T})}(K) = s^\gamma_{\eps /(2\norm{T})}(K)$. Now suppose that $\norm{Q_\ess\dual V_\ess\dual y}^q > \abs{K}^q -(\eps /(8\norm{T}))^q$, where $\ess \subseteq \Upsilon$. Passing to a subnet, we may assume $\norm{Q_\ess\dual V_\ess\dual y_j}^q > \abs{K}^q -(\eps /(8\norm{T}))^q$ for all $j$, hence by the induction hypothesis $T\dual Q_\ess\dual V_\ess\dual y_{f(j)}\in s^{f(j)}_{\eps /4}(T\dual Q_\ess\dual V_\ess\dual K)$ for all $j$. Again, by the cofinality of $f(J)$ in $\gamma$, \begin{equation*} T\dual Q_\ess\dual V_\ess\dual y = w\dual-\lim_j T\dual Q_\ess\dual V_\ess\dual y_{f(j)} \in \bigcap_{j\in J}s^{f(j)}_{\eps /4}(T\dual Q_\ess\dual V_\ess\dual K) = s^\gamma_{\eps /4}(T\dual Q_\ess\dual V_\ess\dual K).\end{equation*} The assertion of the lemma thus passes to limit ordinals, and we are done.
\end{proof}

The final step in our preparation to prove Proposition~\ref{unifac} is to state the following definition and lemma.

\begin{definition}\label{postdoc1}
For real numbers $a\geqslant 0$, $b>c>0$ and $1\leqslant d<\infty$, define
\[
\sigma (a,b,c,d):= \inf \big\{ n\in \nat \, \big\vert\, (b-c)^d(n-1)\geqslant (2a)^d - b^d \big\} \, .
\]
\end{definition}

\begin{lemma}[{\cite[Lemma~2.8]{Brookera}}]\label{postdoc2}
Let $\Lambda$ be a set, $\{ E_\lambda \mid \lambda\in\Lambda \}$ a family of Banach spaces, $1\leqslant q <\infty$, $p$ predual to $q$, $K\subseteq (\bigoplus_{\lambda\in\Lambda}E_\lambda )_p\dual$ a nonempty, $w\dual$-compact set and $\eps>\delta >0$. Suppose $\eta_\delta $ is a nonzero ordinal such that $s_\delta^{\eta_\delta}(U_\eff\dual K) = \emptyset$ for every $\eff\in\Lambda\fin$. Then $s_\eps^{\eta_\delta \cdot \sigma (\vert K\vert , \eps , \delta , q)}(K) = \emptyset$, hence ${\rm Sz}_\eps (K)\leqslant \eta_\delta \cdot \sigma (\vert K\vert , \eps , \delta , q)$.
\end{lemma}

\begin{proof}[Proof of Proposition~\ref{unifac}]
We prove (i)$\Rightarrow$(ii)$\Rightarrow$(iv)$\Rightarrow$(iii)$\Rightarrow$(i), assuming $T\neq 0$ (the result is obvious otherwise).

To see that (i) $\Rightarrow$ (ii), suppose that there is $\eps >0$ such that \begin{equation*} \sup \set{\meeszlenk{\eps}{TU_\eff}\mid \eff\in\Lambda\fin}\geqslant \omega^\alpha .\end{equation*} We want to show that $\mszlenk{T} >\omega^\alpha$, so to this end note that by Lemma~\ref{poopoopoo} we have
\begin{align*}
\meeszlenk{\eps /2}{T} = {\rm Sz}_{\eps /2}\big( T\dual \cball{(\bigoplus_{\upsilon \in \Upsilon}F_\upsilon)_p\dual}\big) &\geqslant \sup \big\{{\rm Sz}_{\eps}\big( U_\eff\dual T\dual \cball{(\bigoplus_{\upsilon \in \Upsilon}F_\upsilon)_p\dual} \big) \,\, \big\vert \,\, \eff\in\Lambda\fin\big\}\\ & = \sup \set{\meeszlenk{\eps}{TU_\eff}\mid \eff\in\Lambda\fin}\\ & \geqslant \omega^\alpha \, .
\end{align*}
As $\meeszlenk{\eps/2}{T}$F cannot be a limit ordinal, it follows that $\mszlenk{T} \geqslant \meeszlenk{\eps/2}{T} >\omega^\alpha$.

We now show (ii) $\Rightarrow$ (iv). Let $\eff \in \Lambda\fin$. Then for $\gee\in\Upsilon\fin$ we have $U_\eff\dual T\dual Q_\gee\dual \cball{(\bigoplus_{\upsilon \in \gee}F_\upsilon)_p\dual}\subseteq U_\eff\dual T\dual \cball{(\bigoplus_{\upsilon \in \Upsilon}F_\upsilon)_p\dual}$, hence $\meeszlenk{\eps}{Q_\gee T U_\eff}\leqslant \meeszlenk{\eps}{TU_\eff}$ for all $\gee\in\Upsilon\fin$ and $\eps >0$. Thus, for each $\eps >0$,
\begin{equation*}
\sup \set{\meeszlenk{\eps}{Q_\gee TU_\eff}\mid \eff\in\Lambda\fin , \, \gee \in \Upsilon\fin} \leqslant \sup \set{\meeszlenk{\eps}{TU_\eff}\mid \eff\in\Lambda\fin},
\end{equation*}
and the implication (ii) $\Rightarrow$ (iv) follows.

Suppose that (iv) holds and fix $\gee\in\Upsilon\fin$. An application of Lemma~\ref{postdoc2} with $K=T\dual Q_\gee\dual \cball{(\bigoplus_{\upsilon\in\Upsilon}F_\upsilon)_p\dual}$, $\delta = \delta(\eps)=\eps/2$ and \[ \eta_{\delta(\eps)} = \sup \{ {\rm Sz}_{\eps/2}(Q_\gee TU_\eff) \mid \eff \in \Lambda\fin, \, \gee\in\Upsilon\fin \} \, \, (<\omega^\alpha)\] yields
\[
{\rm Sz}_\eps (T\dual Q_\gee\dual \cball{(\bigoplus_{\upsilon\in\Upsilon}F_\upsilon)_p\dual}) \leqslant \eta_{\delta(\eps)}\cdot \sigma (\Vert T\Vert , \, \eps, \, \eps/2, \, q)\, .
\]
As $\gee\in\Upsilon\fin$ was arbitrary and $\eta_{\delta(\eps)}$ and $\sigma (\Vert T\Vert, \, \eps, \, \eps/2, \, q)$ do not depend on our choice of $\gee$, we deduce that
\[
\sup \{ {\rm Sz}_\eps (Q_\gee T) \mid \gee\in\Upsilon\fin \} \leqslant \eta_{\delta(\eps)}\cdot \sigma (\Vert T\Vert, \, \eps, \, \eps/2, \, q) <\omega^\alpha\, ,
\]
hence (iv)$\Rightarrow$(iii).

Suppose that (iii) holds. The implication (iii) $\Rightarrow$ (i) will follow from Proposition~\ref{dadeeeee} if we can show that for every $\eps > 0$ there is $\beta_\eps <\omega^\alpha$ and $\delta_\eps \in (0, \, 1)$ with $s^{\beta_\eps}_\eps (T\dual \cball{(\bigoplus_{\upsilon \in \Upsilon}F_\upsilon)_p\dual}) \subseteq \delta_\eps T\dual \cball{(\bigoplus_{\upsilon \in \Upsilon}F_\upsilon)_p\dual}$. If $\eps \geqslant 2\norm{T}$, then $s_\eps (T\dual \cball{(\bigoplus_{\upsilon \in \Upsilon}F_\upsilon)_p\dual}) = \emptyset$, hence $\beta_\eps = 1$ and any $\delta_\eps \in (0, \, 1)$ suffice. So it remains to find suitable $\beta_\eps$ and $\delta_\eps$ for $\eps \in (0, \, 2\norm{T})$. For each $\eps \in (0, \, 2\norm{T})$, let $\xi_\eps=\sup \set{\meeszlenk{\eps}{Q_\gee T}\mid \gee\in\Upsilon\fin}$. As $T\dual Q_\gee\dual  \cball{(\bigoplus_{\upsilon \in \gee}F_\upsilon)_p\dual}=T\dual Q_\gee\dual V_\gee\dual \cball{(\bigoplus_{\upsilon \in \Upsilon}F_\upsilon)_p\dual}$ for each $\gee\in\Upsilon\fin$, it follows that $\sup \set{\meeszlenk{\eps}{V_\gee Q_\gee T}\mid \gee\in\Upsilon\fin}=\xi_\eps$ for each $\eps \in (0, \, 2\norm{T})$. Let $\eps \in (0, \, 2\norm{T})$ be fixed and let $x\in s^{\xi_{\eps/4}}_\eps \big(T\dual \cball{(\bigoplus_{\upsilon \in \Upsilon}F_\upsilon)_p\dual}\big)$ (if $s^{\xi_{\eps/4}}_\eps \big(T\dual \cball{(\bigoplus_{\upsilon \in \Upsilon}F_\upsilon)_p\dual}\big) = \emptyset$, then taking $\beta_\eps = \xi_{\eps /4}$ and any $\delta_\eps \in (0, \, 1)$ will do). Since $s^{\xi_{\eps/4}}_{\eps/4} \big(T\dual Q_\gee\dual V_\gee\dual \cball{(\bigoplus_{\upsilon \in \Upsilon}F_\upsilon)_p\dual}\big) = \emptyset$ for all $\gee\in\Upsilon\fin$, an appeal to Lemma~\ref{baconlemma} gives us $y \in s^{\xi_{\eps /4}}_{\eps /(2\norm{T})}\big( \cball{(\bigoplus_{\upsilon \in \Upsilon}F_\upsilon)_p\dual} \big)$ such that $T\dual y = x$ and
\begin{equation*}
\norm{y}^q = \sup_{\gee\in\Upsilon\fin}\norm{Q_\gee\dual V_\gee\dual y}^q \leqslant 1 - \left(\frac{\eps}{8\norm{T}}\right)^q\, .
\end{equation*}
In particular, since $x\in s^{\xi_{\eps/4}}_\eps \big(T\dual \cball{(\bigoplus_{\upsilon \in \Upsilon}F_\upsilon)_p\dual}\big)$ was arbitrary,
\[
s^{\xi_{\eps/4}}_\eps \Big(T\dual \cball{(\bigoplus_{\upsilon \in \Upsilon}F_\upsilon)_p\dual}\Big) \subseteq \left( 1 - \Big(\frac{\eps}{8\norm{T}}\Big)^q \right)^{1/q}T\dual \cball{(\bigoplus_{\upsilon \in \Upsilon}F_\upsilon)_p\dual}\, .
\]
Taking $\beta_\eps = \xi_{\eps/4}$ and $\delta_\eps = \left( 1 - (\eps /(8\norm{T}))^q \right)^{1/q}$ for each $\eps \in (0, \, 2\norm{T})$ completes the proof.\end{proof}

\begin{corollary}\label{uncfac}
Let $\alpha$ be an ordinal of uncountable cofinality. Then $\szlenkop{\alpha}$ has the factorization property.
\end{corollary}

\begin{proof}
By Theorem~\ref{idealthm} and Theorem~\ref{manna}, it suffices to show that $(\szlenkop{\alpha}, \, \szlenkop{\alpha})$ is a $\Sigma_p$-pair ($1<p<\infty$). To this end, let $(E_m)_m$ and $(F_n)_n$ be sequences of Banach spaces and let $T\in \allop \big( (\bigoplus_{m\in \nat}E_m)_p, \, (\bigoplus_{n\in\nat}F_n)_p \big)$ be such that $Q_\gee TU_\eff \in \szlenkop{\alpha}$ for all $\eff, \, \gee \in \nat\fin$. By Proposition~\ref{collection}(iii), for each pair $(\eff , \, \gee)\in \nat\fin\times\nat\fin$ there is $\alpha(\eff , \, \gee)\leqslant \alpha$ such that ${\rm Sz}(Q_\gee TU_\eff ) = \omega^{\alpha(\eff , \, \gee)}$. However, since
\begin{equation*}
\cf{\alpha(\eff, \, \gee)}\leqslant cf \Big(\omega^{\alpha(\eff, \, \gee)}\Big) = cf \bigg( \sup_{n\in\nat}\meeszlenk{1/n}{Q_\gee TU_\eff}\bigg) = \omega <\omega_1 \leqslant \cf{\alpha},
\end{equation*}
it must be that $\alpha(\eff, \, \gee) <\alpha$ for each $(\eff, \, \gee)\in \nat\fin \times \nat \fin$. Consider the ordinal $\alpha' = \sup \set{\alpha (\eff, \, \gee)\mid \eff, \, \gee\in\nat\fin}$. We have $\alpha'\leqslant \alpha$ and, since $\nat\fin \times \nat\fin$ is countable, $\cf{\alpha'}$ is countable also, hence $\alpha' <\alpha$. As $\alpha$ is of uncountable cofinality, it is also a limit ordinal, hence $\alpha'+1<\alpha$. Moreover,
\begin{equation*}
\forall \eps >0 \quad \sup\negthinspace \big\{\meeszlenk{\eps}{Q_\gee TU_\eff}\, \big\vert\, \eff, \, \gee\in\nat\fin\big\} \leqslant \omega^{\alpha'}<\omega^{\alpha'+1},
\end{equation*}
and so Proposition~\ref{unifac} yields $T \in \szlenkop{\alpha'+1}\subseteq \szlenkop{\alpha}$. We have thus shown that $(\szlenkop{\alpha}, \, \szlenkop{\alpha})$ is a $\Sigma_p$-pair, which completes the proof.\end{proof}

The following is open:
\begin{problem}\label{factortractor}
Let $\alpha$ be an ordinal. Are the following are equivalent?

\nr{i} $\alpha$ is of uncountable cofinality.

\nr{ii} $\szlenkop{\alpha}$ has the factorization property.
\end{problem}

With Corollary~\ref{uncfac} we have already just established the implication (i)$\Rightarrow$(ii) of Problem~\ref{factortractor}. The remainder of this paper is motivated by the search for a proof of the reverse implication (ii)$\Rightarrow$(i). Although we do not obtain the full answer, we obtain some partial results and anticipate that further development of the ideas presented here may eventually lead to a complete solution. Note that Theorem~\ref{manna} does not yield any further information on the classification of those ordinals $\alpha$ for which $\szlenkop{\alpha}$ has the factorization property since, as noted later in Remark~\ref{julyjulyjuly}, $(\szlenkop{\alpha}, \, \szlenkop{\alpha})$ fails to be a $\Sigma_p$-pair for any $1<p<\infty$ whenever $\alpha$ is an ordinal of countable cofinality.

\section{Counterexamples to the factorization property}\label{theredethere}

Our goal in this section is to prove the following theorem.
\begin{theorem}\label{nonfactor}
Let $\beta$ be an ordinal of countable cofinality. Then $\szlenkop{\omega^\beta}$ does not have the factorization property.
\end{theorem}

One of the main ingredients in our construction of counterexamples to the factorization property is the following result concerning the submultiplicity of the $\eps$-Szlenk index of a Banach space, due to G. Lancien.

\begin{proposition}[{\cite[p.212]{Lancien2006}}]\label{mubsult}
Let $E$ be a Banach space and $\eps, \, \eps' >0$. Then \begin{equation*} \meeszlenk{\eps \eps'}{E}\leqslant \meeszlenk{\eps}{E} \cdot \meeszlenk{ \eps'}{E} \, .\end{equation*}
\end{proposition}

Some remarks concerning the use of Proposition~\ref{mubsult} are in order. Suppose that $E$ is an infinite-dimensional Asplund space and let $\alpha$ denote the unique ordinal satisfying $\mszlenk{E}=\omega^{\alpha}$. The submultiplicity of the $\eps$-Szlenk index seems to be of use in analysis of $E$ only in the case that the ordinal $\omega^{\alpha}$ is closed under ordinal multiplication, which is true if and only if $\alpha$ is closed under ordinal addition, which is true if and only if $\alpha = \omega^\beta$ for some ordinal $\beta$. Indeed, suppose that $\alpha$ is \emph{not} a power of $\omega$; then there is $\gamma <\alpha$ such that $\gamma \cdot 2 \geqslant \alpha$. Let $\eps$ be so small that $\meeszlenk{\eps}{E}\geqslant \omega^\gamma$. Then for $0<\eps' , \eps'' \leqslant \eps $ we have
\begin{equation*}
\meeszlenk{\eps' \eps''}{E} \leqslant \omega^\alpha \leqslant \omega^{\gamma \cdot 2}\leqslant \meeszlenk{\eps'}{E} \cdot \meeszlenk{\eps''}{E},
\end{equation*}
so that submultiplicity of the $\eps$-Szlenk index of $E$ is essentially trivial in this case. In particular, in this case the submultiplicity of the $\eps$-Szlenk index of $E$ does not yield any information regarding the growth of $\meeszlenk{\eps}{E}$ as $\eps$ goes to zero. By contrast, if $\alpha=\omega^\beta$ for some $\beta$, then it is possible to use the submultiplicity of the $\eps$-Szlenk index to obtain a certain growth condition on $\meeszlenk{\eps}{E}$, and similar growth conditions on the $\eps$-Szlenk indices of operators in $\Op (\szl_\alpha)$ (see Proposition~\ref{tunkyfown} below). By constructing an element of $\szlenkop{\alpha}$ that cannot satisfy any such growth condition, we will show that the containment $\Op (\szl_\alpha)\subseteq \szlenkop{\alpha}$ is proper.

\begin{definition}
Let $\beta$ be an ordinal of countable cofinality. A sequence $(\beta_n)_{n\in \nat}$ in $\omega^\beta$ is called a \emph{superadditive cofinal sequence for $\omega^\beta$} if $\set{\beta_n \mid n\in \nat}$ is cofinal in $\omega^\beta$ and $\beta_{n_1}+\beta_{n_2}\leqslant \beta_{n_1+n_2}$ for all $n_1, \, n_2 \in \nat$ (including when $n_1 = n_2$).
\end{definition}

It is easy to see that each ordinal $\beta$ of countable cofinality admits a superadditive cofinal sequence for $\omega^\beta$. Indeed, for such an ordinal $\beta$ we have that $\omega^\beta$ is also of countable cofinality, and so we may choose a sequence $(\gamma_m)_{m\in\nat}$ in $\omega^\beta$ such that $\set{\gamma_m\mid m\in \nat}$ is cofinal in $\omega^\beta$. It is then straightforward to inductively define a strictly increasing sequence $(m_n)_{n\in \nat}$ in $\nat$ such that $(\beta_n = \gamma_{m_n})_{n\in \nat}$ is a superadditive cofinal sequence for $\omega^\beta$.

For a nonzero ordinal $\beta$ of countable cofinality, the following proposition establishes a necessary condition for membership of the operator ideal $\Op (\szl_{\omega^\beta})$, and will be used in the proof of Theorem~\ref{nonfactor}. The proposition asserts that elements of $\Op (\szl_{\omega^\beta})$ must possess a certain restricted-growth property defined in terms of arbitrary superadditive cofinal sequences for $\omega^\beta$.

\begin{proposition}\label{tunkyfown}
Let $\beta$ be a nonzero ordinal of countable cofinality. For each $T\in \Op (\szl_{\omega^\beta})$ and superadditive cofinal sequence for $\omega^\beta$, $(\beta_n)_{n\in \nat}$ say, there exists $n_0\in\nat$ such that\begin{equation*}
\meeszlenk{1/2^n}{T}\leqslant \omega^{\beta_{n_0\cdot n}}
\end{equation*} for all $n\in\nat$.
\end{proposition}

\begin{proof}
The result if trivial if $T=0$, so we assume henceforth that $T\neq 0$. Let $D$, $E$ and $F$ be Banach spaces and $A\in \allop (E, \, D)$ and $B\in \allop (D, \, F)$ operators such that $D\in \szl_{\omega^\beta}$, $T=BA$ and, without loss of generality, $\norm{B} \leqslant 1$. The bound $\norm{B}\leqslant 1$ and Lemma~\ref{poopoopoo} ensure that \begin{equation}\label{mablefable} \forall \, \eps >0 \quad \meeszlenk{\eps}{T}\leqslant \meeszlenk{\eps}{A}\leqslant \meeszlenk{\eps / (2\norm{A})}{D} \, .\end{equation}

Let $s \in \big\{n\in \nat \mid \meeszlenk{1/(2\norm{A})}{D}\leqslant \omega^{\beta_n}\big\}$, $t\in \set{n\in \nat \mid \meeszlenk{1/2}{D}\leqslant \omega^{\beta_n}}$ and set $n_0 = s+t$ (our assumption that $\beta>0$ guarantees the existence of such $s$ and $t$). By Proposition~\ref{mubsult} and (\ref{mablefable}), for each $n\in \nat$ we have
\begin{align*}
\meeszlenk{1/2^n}{T}\leqslant \meeszlenk{1/(2^{n+1}\Vert A\Vert)}{D} \leqslant \meeszlenk{1/(2\norm{A})}{D} \cdot \big(\, \meeszlenk{1/2}{D} \big)^n &\leqslant \omega^{\beta_s} \cdot \omega^{\beta_t \cdot n} \\ &\leqslant \omega^{\beta_{s+tn}}\\ &\leqslant \omega^{\beta_{n_0\cdot n}}\, . \qedhere
\end{align*}
\end{proof}

For the remainder of this section, let $r=0$ or $1<r<\infty$ be fixed. We now detail a construction (inspired by Szlenk's construction in \cite{Szlenk1968} of a family of Banach spaces whose Szlenk indices are unbounded in $\omega_1$) that takes a given operator $T$ and yields an operator $T_{\alpha}$ for each ordinal $\alpha$ in such a way that $T_{0} =T$ and $T_{\alpha}$ is obtained via direct sums of predecessors in the construction. For Banach spaces $D$ and $G$ and an operator $S\in \allop(D, \, G)$, we write $S_{[n]} = (\bigoplus_{i=1}^nS)_1$ for each $n\in \nat$ and $S_{+} = (\bigoplus_{n\in\nat}S_{[n]})_r$. 

\begin{construction}\label{szlenkvar} For Banach spaces $E$ and $F$ and $T\in \allop (E, \, F)$, define $T_{0} = T$, $T_{\alpha+1} = (T_{\alpha})_{+}$ for each ordinal $\alpha$ and $T_{\alpha} = (\bigoplus_{\xi <\alpha}T_{\xi})_r$ whenever $\alpha$ is a limit ordinal. 
\end{construction}

With respect to Construction~\ref{szlenkvar}, note that $\norm{T_{\alpha}} = \norm{T}$ for all ordinals $\alpha$. Our counterexamples to the factorization property shall be obtained as direct sums of operators obtained via this construction. For this we shall require some estimates on the Szlenk and $\eps$-Szlenk indices of the operators $T_\alpha$ in terms of $\mszlenk{T}$. 

For a noncompact Asplund operator $T$, let $\alpha_T$ denote the unique ordinal satisfying $\mszlenk{T} = \omega^{\alpha_T}$. Then we may write $\alpha_T = \eta_T +\omega^{\zeta_T}$, where $\zeta_T$ is uniquely determined by the Cantor normal form of $\alpha_T$ and $\eta_T = \inf \set{\eta \mid \alpha_T= \eta + \omega^{\zeta_T}}$. The following result gives the required estimates of $\mszlenk{T_\alpha}$ and $\meeszlenk{\eps}{T_\alpha}$, $\alpha \in \ord$.

\begin{proposition}\label{babysitter}
Let $T$ be a noncompact Asplund operator.

\nr{i} Suppose $\eps >0$ is so small that $\meeszlenk{\eps}{T}>\omega^{\eta_T}$. Then $\meeszlenk{\eps}{T_\alpha}>\omega^{\eta_T +\alpha}$ for every ordinal $\alpha$.

\nr{ii} $\mszlenk{T_{\alpha}} = \mszlenk{T}$ for all $\alpha<\omega^{\zeta_T}$.
\end{proposition}

To prove part (i) of Proposition~\ref{babysitter}, we require the following lemma.

\begin{lemma}\label{samsung}
Let $E_1, \ldots , E_n$ be Banach spaces and $K_1 \subseteq E_1\dual, \ldots, K_n \subseteq E_n\dual$ $w\dual$-compact sets. Consider $\prod_{i=1}^n K_i$ as a subset of $(\bigoplus_{i=1}^n E_i)_1\dual  = (\bigoplus_{i=1}^n E_i\dual)_\infty$. Then for all $\eps > 0$, ordinals $\alpha $ and $1\leqslant j \leqslant n$,
\begin{equation}\label{wormy}
\textstyle K_1 \times \ldots \times K_{j-1} \times \mdset{\alpha}{\eps}{K_j} \times K_{j+1} \times \ldots \times K_n \subseteq \mdset{\alpha}{\eps}{\prod_{i=1}^n K_i}.
\end{equation}
It follows that for all $\eps > 0$ and ordinals $\alpha$,
\begin{equation}\label{wormfriend}
\textstyle \prod_{i=1}^n \mdset{\alpha}{\eps}{K_i} \subseteq \mdset{\alpha \cdot n}{\eps}{\prod_{i=1}^n K_i}.
\end{equation}
\end{lemma}

\begin{proof}
We prove (\ref{wormy}), with (\ref{wormfriend}) then following from $n$ applications of (\ref{wormy}). Trivially, (\ref{wormy}) holds for $\alpha =0$. We now suppose that $\beta$ is an ordinal such that (\ref{wormy}) holds for $\alpha =\beta$, and show that (\ref{wormy}) then holds for $\alpha = \beta +1$. Fix $j\in \set{1, \ldots , n}$. Let $(k_1, \ldots , k_n ) \in \prod_{i=1}^n K_i$ be such that $k_j \in \mdset{\beta+1}{\eps}{K_j}$ (if $\mdset{\beta+1}{\eps}{K_j}$ is empty then we are done) and let $V \ni (k_1, \ldots , k_n )$ be $w\dual$-open. Then there are $w\dual$-open sets $V_i \subseteq E_i\dual$, $1\leqslant i \leqslant n$, such that $(k_1, \ldots , k_n ) \in V_1 \times \ldots \times V_n \subseteq V$. For $1 \leqslant l \leqslant m \leqslant n$ we shall write $K^{l,\, m} = \prod_{i=l}^m K_i$ and $W^{l,\, m} = \prod_{i=l}^m (V_i\cap K_i)$. Assuming $1< j < n$ (the argument for the other two cases being similar), we have
\begin{eqnarray*} \textstyle
\diam \big(V\cap \mdset{\beta}{\eps}{\prod_{i=1}^n K_i} \big) &\geqslant&
\textstyle \diam \big(\big(\prod_{i=1}^n V_i \big)\cap \big(K^{1,\, j-1} \times \mdset{\beta}{\eps}{K_j} \times K^{j+1,\, n}\big)\big) \\
&=&\textstyle  \diam \big(W^{1,\, j-1} \times \big(V_j \cap \mdset{\beta}{\eps}{K_j}\big) \times W^{j+1,\, n} \big) \\
&\geqslant & \diam \big(V_j \cap \mdset{\beta}{\eps}{K_j}\big)\\ &>& \eps .
\end{eqnarray*}
It follows that $(k_1 , \ldots , k_n) \in \mdset{\beta +1}{\eps}{\prod_{i=1}^n K_i}$, thus (\ref{wormy}) holds for $\alpha = \beta+1$.

Now suppose that $\beta $ is a limit ordinal and that (\ref{wormy}) holds for every $\alpha < \beta$. Assuming once again, for notational convenience, that $1< j< n$, we have
\begin{eqnarray*}
K^{1, \, j-1} \times \mdset{\beta}{\eps}{K_j} \times K^{j+1, \, n}
&=&\textstyle K^{1, \, j-1} \times ( \bigcap_{\alpha< \beta} \mdset{\alpha}{\eps}{K_j} ) \times K^{j+1, \, n}\\
&=&\textstyle \bigcap_{\alpha < \beta} ( K^{1, \, j-1} \times \mdset{\alpha}{\eps}{K_j} \times K^{j+1, \, n}) \\
&\subseteq& \textstyle \bigcap_{\alpha < \beta} \mdset{\alpha}{\eps}{\prod_{i=1}^n K_i} \\
&=&\textstyle \mdset{\beta}{\eps}{\prod_{i=1}^n K_i}.
\end{eqnarray*}
The inductive proof is now complete.
\end{proof}

\begin{remark} The reverse inclusion to (\ref{reedideas}) also holds; this is achieved by substituting $\omega^\alpha$ in place of $\alpha$ in (\ref{wormy}) (see the statement of  Lemma~\ref{carbon}). \end{remark}

To prove Proposition~\ref{babysitter}(i), we fix $\eps$ and proceed via transfinite induction on $\alpha$. Part (i) is trivially true for $\alpha =0$. So suppose that (i) holds for some $\alpha = \gamma$; we show that it then holds for $\alpha = \gamma+1$. We have ${\rm Sz}_{\eps}\big( ({T_\gamma})_{[n]}\big)> \omega^{\eta_T +\gamma}\cdot n$ for all $n\in \nat$ by Lemma~\ref{samsung}, hence
\begin{equation*}
{\rm Sz}_{\eps}( T_{\gamma+1}) \geqslant \sup_{n\in\nat} {\rm Sz}_{\eps}\big( ({T_\gamma})_{[n]}\big) \geqslant \sup_{n\in \nat}\omega^{\eta_T +\gamma}\cdot n = \omega^{\eta_T +\gamma+1}\, .
\end{equation*}
As ${\rm Sz}_{\eps}( T_{\gamma+1})$ cannot be a limit ordinal, we conclude that ${\rm Sz}_{\eps}( T_{\gamma+1})> \omega^{\eta_T +\gamma+1}$. In particular, assertion (i) of Proposition~\ref{babysitter} passes to successor ordinals.

Now suppose that $\gamma$ is a limit ordinal and that assertion (i) of Proposition~\ref{babysitter} holds for all $\alpha<\gamma$. Then
\begin{equation*}
{\rm Sz}_{\eps}( T_{\gamma}) \geqslant \sup_{\alpha <\gamma} {\rm Sz}_{\eps}( {T_\alpha}) \geqslant \sup_{\alpha<\gamma}\omega^{\eta_T +\alpha} = \omega^{\eta_T +\gamma},
\end{equation*}
hence ${\rm Sz}_{\eps}( T_{\gamma}) > \omega^{\eta_T +\gamma}$. This concludes the inductive proof of Proposition~\ref{babysitter}(i).

The proof of assertion (ii) of Proposition~\ref{babysitter} will take substantially more effort. We proceed via a sequence of lemmas, giving proofs as necessary. We must first make a note of some convenient notation. For a set $\Lambda$, a family of Banach spaces $\set{E_\lambda \mid \lambda \in \Lambda}$, a family $\set{K_\lambda \subseteq E_\lambda\dual \mid \lambda \in \Lambda}$ of nonempty, absolutely convex, $w\dual$-compact sets satisfying $\sup_{\lambda \in \Lambda} \abs{K_\lambda} <\infty$, and $1\leqslant q <\infty$, we define
\begin{equation*}
B_q(K_\lambda \mid \lambda \in \Lambda):= \bigcup_{(a_\lambda)_{\lambda \in \Lambda} \in \cball{\ell_q(\Lambda)}} \prod_{\lambda\in\Lambda}a_\lambda K_\lambda \, ,
\end{equation*}
and always consider $B_q(K_\lambda \mid \lambda \in \Lambda)$ as a subset of $(\bigoplus_{\lambda \in \Lambda}E_\lambda)_p\dual$, where $p$ is predual to $q$. Such a set $B_q(K_\lambda \mid \lambda \in \Lambda)$ so defined is bounded and $w\dual$-compact. Indeed, if for each $\lambda \in \Lambda$ we let $T_\lambda : E_\lambda \To C((K_\lambda, \, w\dual))$ denote the operator that sends $x\in E_\lambda$ to the $w\dual$-continuous map $K_\lambda \To \field : k\mapsto \langle k, \, x\rangle$, then $B_q(K_\lambda \mid \lambda \in \Lambda)=(\bigoplus_{\lambda\in\Lambda}T_\lambda)_p\dual \cball{(\bigoplus_{\lambda\in\Lambda}C((K_\lambda, \, w\dual)))_p\dual}$.

Our immediate goal is to establish the following lemma.
\begin{lemma}\label{oldfrount}
Let $E_1, \ldots , E_n$ be Banach spaces, $K_1 \subseteq E_1\dual , \ldots , K_n\subseteq E_n\dual$ nonempty, absolutely convex, $w\dual$-compact sets, $\eps >0$, $\alpha$ a nonzero ordinal and $1\leqslant q < \infty$. If $s^{\omega^\alpha}_\eps (B_q(K_i\mid 1\leqslant i \leqslant n)) \neq \emptyset$, then for every $\delta \in (0, \, \eps)$ there is $i \leqslant n$ such that $\mdset{\omega^\alpha}{\delta}{K_i} \neq \emptyset$.
\end{lemma}

The proof of Lemma~\ref{oldfrount} is similar to that of \cite[Lemma~2.5]{Brookera} (though neither Lemma~\ref{oldfrount} above or \cite[Lemma~2.5]{Brookera} are strong enough to be used in place of the other). To prove Lemma~\ref{oldfrount}, we require the following three preliminary results; the first two sublemmas are proved in \cite[Lemma~3.5 and Lemma~3.1]{Brookera}, and the third we shall prove here.

\begin{sublemma}\label{lecondsast}
Let $E_1, \ldots , E_n$ be Banach spaces, $K_1\subseteq E_1\dual, \ldots , K_n\subseteq E_n\dual$ nonempty, absolutely convex, $w\dual$-compact sets, $1\leqslant q<\infty$ and $l\in\nat$. Let $L=\nat^n\cap (l+n^{1/q})\cball{\ell_q^n}$. Then\begin{equation*}  B_q(K_i\mid 1\leqslant i \leqslant n) \subseteq \bigcup_{(k_i)_{i=1}^n\in L} \,\prod_{i=1}^n \frac{k_i}{l}K_i \, . \end{equation*}
\end{sublemma}

\begin{sublemma}\label{unionlemma}
Let $E$ be a Banach space, $K_1, \ldots , K_n \subseteq E\dual$ $w\dual$-compact sets  and $\eps >0$. Let $\alpha$ be an ordinal and $m< \omega$. Then:

\nr{i} $\mdset{mn}{\eps}{\bigcup_{i=1}^nK_i} \subseteq \bigcup_{i=1}^n \mdset{m}{\eps}{K_i}$.

\nr{ii} If $\alpha$ is a limit ordinal, then $\mdset{\alpha}{\eps}{\bigcup_{i=1}^nK_i} \subseteq \bigcup_{i=1}^n \mdset{\alpha}{\eps}{K_i}$.
\end{sublemma}

\begin{sublemma}\label{carbonesque}
Let $E_1, \ldots , E_n$ be Banach spaces and $K_1 \subseteq E_1\dual, \ldots , K_n \subseteq E_n\dual$ nonempty $w\dual$-compact sets. Let $1\leqslant q<\infty$ and $a_1 , \ldots , a_n \geqslant 0$ be real numbers such that $\sum_{i=1}^na_i^q \leqslant 1$. Let $p$ be predual to $q$ and consider $\prod_{i=1}^na_iK_i$ as a subset of $(\bigoplus_{i=1}^n E_i)_p\dual$. Then, for all $\eps >0$ and ordinals $\alpha$,
\begin{equation}
s^{\omega^\alpha}_{\eps}\bigg( \prod_{i=1}^na_iK_i\bigg) \subseteq  \bigcup_{\substack{g_1, \ldots , g_n <\,  \omega \\ g_1 + \ldots + g_n = \,1}} \hspace{0.5mm}\prod_{i=1}^na_i\mdset{\omega^\alpha \cdot g_i}{\eps}{K_i}\, .
\end{equation}
\end{sublemma}

\begin{proof}
We give the proof for the case $n=2$, the proof of the general case being similar. To minimize the use of subscripts, let $a=a_1$, $b=a_2$, $K=K_1$ and $L=K_2$; our goal is thus to show that for all ordinals $\alpha$ and $\eps>0$:
\begin{align}\label{aboutto}
s_\eps^{\omega^\alpha}(aK\times bL)\subseteq (as_\eps^{\omega^\alpha}(K)\times bL)\cup (aK\times bs_\eps^{\omega^\alpha}(L))\, .
\end{align}
We fix $\eps>0$ and proceed via induction on $\alpha$. To establish (\ref{aboutto}) for $\alpha=0$, let $k\in K$ and $l\in L$ be such that
\[
(ak, \, bl) \in (aK\times bL)\setminus ((as_\eps(K)\times bL)\cup (aK\times bs_\eps(L)))\, .
\]
Then there exist $w\dual$-open sets $U\ni k$ and $V\ni l$ such that $\diam (K\cap U)\leqslant \eps$ and $\diam (L\cap V)\leqslant \eps$. The $w\dual$-neighborhood $W:= aU\times bV \ni (ak, \, bl)$ satisfies
\[
\diam ((aK\times bL)\cap W)\leqslant \big((a\cdot \diam(K\cap U))^q+(b\cdot \diam (L\cap V))^q\big)^{1/q} \leqslant \eps\, ,
\]
hence $(ak, \, bl)\notin s_\eps (aK\times bL)$, as required.

Suppose that $\beta$ is an ordinal such that (\ref{aboutto}) holds for $\alpha = \beta$. For each $j\in\nat$ let $m_j = \sum_{t=1}^jt$. To establish (\ref{aboutto}) for $\alpha = \beta+1$, we first show that for all $j\in \nat$,
\begin{align}\label{dabumbumbum}
s_\eps^{\omega^\beta\cdot m_j}(aK\times bL) \subseteq \bigcup_{h=0}^j as_\eps^{\omega^\beta \cdot h}(K) \times bs_\eps^{\omega^\beta\cdot (j-h)}(L)\, .
\end{align}
For $j=1$, (\ref{dabumbumbum}) is the induction hypothesis that (\ref{aboutto}) holds for $\alpha=\beta$. For $j\in\nat$ such that (\ref{dabumbumbum}) holds, it follows by Sublemma~\ref{unionlemma} ((i) if $\beta=0$, (ii) if $\beta >0$) that
\begin{align*}
s_\eps^{\omega^\alpha\cdot m_{j+1}}(aK\times bL) &\subseteq s_\eps^{\omega^\beta\cdot (j+1)}\bigg( \bigcup_{h=0}^jas_\eps^{\omega^\beta\cdot h}(K)\times bs_\eps^{\omega^\beta\cdot(j-h)}(L) \bigg)\\
&\subseteq \bigcup_{h=0}^j s_\eps^{\omega^\beta} (as_\eps^{\omega^\beta\cdot h}(K)\times bs_\eps^{\omega^\beta \cdot (j-h)}(L))\\
&\subseteq \bigcup_{j=0}^{j+1}as_\eps^{\omega^\beta\cdot h}(K)\times bs_\eps^{\omega^\beta \cdot (j+1-h)}(L)\, ,
\end{align*}
hence (\ref{dabumbumbum}) holds for all $j\in \nat$ (by induction on $j$).

By (\ref{dabumbumbum}), for each $(x, \, y)\in s_\eps^{\omega^{\beta+1}}(aK\times bL) = \bigcap_{j\in\nat}s_\eps^{\omega^\beta\cdot m_j}(aK\times bL)$ we have:
\[
\forall \, m<\omega \quad \exists\,  i(m)\leqslant m \quad x\in s_\eps^{\omega^\beta \cdot i(m)}(K), \, y\in s_\eps^{\omega^\beta \cdot (m-i(m))}(L)\, .
\]
If $(i(m))_{m<\omega}$ is unbounded in $\omega$ then $x\in as_\eps^{\omega^{\beta+1}}(K)$, otherwise $(m-i(m))_{m<\omega}$ is unbounded in $\omega$ and $y\in bs_\eps^{\omega^{\beta+1}}(L)$. It follows that (\ref{aboutto}) passes to successor ordinals.

If $\beta$ is a limit ordinal and (\ref{aboutto}) holds for all $\alpha<\beta$, then a similar argument to that used above shows that for $(x, \, y) \in s_\eps^{\omega^\beta}(aK\times bL) = \bigcap_{\alpha<\beta}s_\eps^{\omega^\alpha}(aK\times bL)$ we have either $x\in as_\eps^{\omega^\beta}(K)$ or $y\in bs_\eps^{\omega^\beta}(L)$. In particular, (\ref{aboutto}) passes to limit ordinals. The inductive proof is now complete.
\end{proof}

\begin{proof}[Proof of Lemma~\ref{oldfrount}]
Fix $\delta \in (0, \, \eps)$. Let $l \geqslant \delta n^{1/q}(\eps -\delta)^{-1}$ be an integer and let $ L = \nat^n\cap (l+n^{1/q})\cball{\ell_q^n} $. By Sublemma~\ref{lecondsast} and the hypothesis of Lemma~\ref{oldfrount},
\begin{equation*}
s^{\omega^\alpha}_{\eps}\hspace{-0.5mm}\Bigg(\bigcup_{(k_i)\in L} \prod_{i=1}^n \frac{k_i}{l}K_i \Bigg) \supseteq s^{\omega^\alpha}_{\eps}(B_q(K_i\mid 1\leqslant i \leqslant n)) \supsetneq \emptyset\, .
\end{equation*}
Thus, since $L$ is finite and $\omega^\alpha$ is a limit ordinal, Sublemma~\ref{unionlemma}(ii) ensures the existence of $(h_i)_{i=1}^n\in L$ such that
\begin{equation}\label{weednaughter}
s^{\omega^\alpha}_{\eps}\bigg(\prod_{i=1}^n \frac{h_i}{l}K_i \bigg) \neq \emptyset \, .
\end{equation}
Let $\rho = (1+ \frac{n^{1/q}}{l})^{-1}$. By (\ref{weednaughter}) and the homogeneity of the derivations $s^\gamma_{\eps'}$ (where $\gamma$ is an ordinal and $\eps'>0$), we have
\begin{equation}\label{peedpaughter}
s^{\omega^\alpha}_{\rho \eps}\bigg(\prod_{i=1}^n \frac{\rho h_i}{l}K_i \bigg) = \rho s^{\omega^\alpha}_{\eps}\bigg(\prod_{i=1}^n \frac{h_i}{l}K_i \bigg) \neq \emptyset \, .
\end{equation}
Thus, since $ ||(\frac{\rho h_i}{l})_{i=1}^n||_{\ell_q^n} \leqslant 1$, it follows from (\ref{peedpaughter}) and Sublemma~\ref{carbonesque} that there is $i \leqslant n$ such that $\mdset{\omega^\alpha}{\rho \eps}{K_{i}}\neq \emptyset$. As $\rho\eps \geqslant\delta$, we have $\mdset{\omega^\alpha}{\delta}{K_{i}}\supseteq \mdset{\omega^\alpha}{\rho\eps}{K_{i}}\supsetneq \emptyset$.
\end{proof}

We will now prove Proposition~\ref{babysitter}(ii). Let $T$ be a noncompact Asplund operator, with $\mszlenk{T} = \omega^{\alpha_T}$ (c.f. the paragraph preceding Proposition~\ref{babysitter}). If $\alpha_T$ is a successor ordinal, then $\omega^{\zeta_T} = 1$, hence (ii) holds in this case since $T_0 = T$. 

Suppose $\alpha_T$ is a limit ordinal. For each $\eps >0$, let $\beta_\eps = \inf \set{\beta \mid \meeszlenk{\eps}{T}<\omega^\beta}$ and $\nu_\eps =\inf \set{\beta_\delta \mid 0< \delta< \eps}$ (note that $\beta_\eps$ and $ \nu_\eps$ exist for all $\eps >0$ since the set $\set{\omega^\beta \mid \beta<\alpha_T}$ is cofinal in $\omega^{\alpha_T}$). Our immediate goal is to show the following:
\begin{equation}\label{warnchood}
\forall \eps >0\quad  \forall \alpha \in \ord \quad {\rm Sz}_\eps (T_\alpha) < \omega^{\nu_\eps + \alpha +1} \, .
\end{equation}
We proceed by induction on $\alpha$. For $\eps >0$ we have \begin{equation*} {\rm Sz}_\eps (T_0) <\omega^{\beta_\eps} \leqslant \omega^{\nu_\eps}<\omega^{\nu_\eps +0+1}\, ,\end{equation*} hence the estimate of (\ref{warnchood}) holds for $\alpha =0$ and all $\eps >0$.

Now suppose that $\gamma$ is an ordinal such that the estimate of (\ref{warnchood}) holds for $\alpha = \gamma$ and all $\eps >0$; we will show that it then holds for $\alpha = \gamma+1$ and all $\eps >0$. By the induction hypothesis, for every $\eps >0$ we have ${\rm Sz}_\eps (T_\gamma)< \omega^{\nu_\eps +\gamma +1}$. It follows then by Lemma~\ref{carbon} that
\[
\forall \eps>0 \quad \forall n\in\nat \quad {\rm Sz}_\eps \big( (T_\gamma)_{[n]} \big) < \omega^{\nu_\eps +\gamma +1} \, .
\]
Thus, Lemma~\ref{oldfrount} yields
\begin{equation}\label{gicken}
\forall \eps > \rho >0 \quad \forall \eff \in \nat\fin \quad {\rm Sz}_\eps \Big( \big(\textstyle  \bigoplus_{n\in\eff}(T_\gamma)_{[n]} \big)_r \Big) < \omega^{\nu_\rho +\gamma+1} \, .
\end{equation}
Moreover, (\ref{gicken}) implies that
\begin{align}\label{postdoc7}
\forall \eps &> \rho >0 \quad \forall \eff \in \nat\fin \notag \\ & {\rm Sz}_\eps \Big( \big(\textstyle  \bigoplus_{n\in\eff}(T_\gamma)_{[n]} \big)_r \Big) \leqslant {\rm Sz}_{(\eps +\rho)/2} \Big( \big(\textstyle  \bigoplus_{n\in\eff}(T_\gamma)_{[n]} \big)_r \Big) < \omega^{\nu_\rho +\gamma+1} \, .
\end{align}
Let $D$ denote the domain of $T_{\gamma+1}$ and let $K=T_{\gamma+1}\dual \cball{D\dual}$, so that $s_{(\eps+\rho)/2}^{\omega^{\nu_\rho +\gamma +1}}(U_\eff\dual K)$ is empty for every $\eff\in\nat\fin$ by (\ref{postdoc7}) (here $U_\eff$ denotes the canonical embedding of the $\ell_r$-direct sum of the domains of the operators $(T_\gamma)_{[n]}$, $n\in\eff$, into the $\ell_r$-direct sum of the domains of the operators $(T_\gamma)_{[n]}$, $n\in\nat$). It follows then by an application of Lemma~\ref{postdoc2} with $\delta = (\eps+\rho)/2$ and $\eta_\delta = \omega^{\nu_\rho +\gamma +1}$ that
\begin{align}\label{postdoc4500}
\forall \eps > \rho& >0 \quad {\rm Sz}_\eps (T_{\gamma+1})\leqslant \omega^{\nu_\rho +\gamma +1} \cdot \sigma \big(\Vert T \Vert, \, \eps, \, (\eps +\rho)/2, \, r(r-1)^{-1}\big)\, .
\end{align}
For each $\eps >0$ there exists $\pi (\eps ) \in (0, \, \eps )$ such that $\nu_{\pi (\eps)} = \inf \{\nu_\rho \mid 0<\rho<\eps\}$. We have
\begin{equation}\label{gow}
\nu_{\pi (\eps )}  = \inf_{\rho \in (0, \, \eps )}\nu_\rho = \inf_{\rho\in(0, \, \eps)}\,  \inf_{\tau \in (0, \, \rho)}\beta_\tau = \inf_{\rho\in (0, \, \eps )}\beta_\rho = \nu_\eps \, ,
\end{equation}
and so from (\ref{gow}) and (\ref{postdoc4500}) (with $\rho = \pi (\eps )$) we have, for every $\eps >0$,
\[
{\rm Sz}_\eps ( T_{\gamma+1} ) < \omega^{\nu_\eps +\gamma +1} \cdot \sigma \big(\Vert T \Vert, \, \eps, \, (\eps +\pi(\eps))/2, \, r(r-1)^{-1}\big) < \omega^{\nu_\eps +(\gamma +1)+1}\, .
\]
In particular, the estimate of (\ref{warnchood}) passes to successor ordinals for every $\eps >0$.

Let $\gamma$ be a limit ordinal and suppose that the estimate of (\ref{warnchood}) holds for every $\alpha <\gamma$ and $\eps >0$. By Lemma~\ref{oldfrount} we have
\begin{equation}\label{postdoc9}
\forall \eps >\rho >0 \quad \forall \eff\in\gamma\fin \quad {\rm Sz}_\eps \Big( \big(\textstyle  \bigoplus_{\alpha\in\eff}{T_\alpha} \big)_r \Big) < \omega^{\nu_\rho +(\max \eff )+1} < \omega^{\nu_\rho +\gamma} \, .
\end{equation}
Moreover, (\ref{postdoc9}) implies that
\begin{align}\label{postdoc10}
\forall \eps >\rho >0& \quad \forall \eff\in\gamma\fin \notag \\ & {\rm Sz}_\eps \Big( \big(\textstyle  \bigoplus_{\alpha\in\eff}{T_\alpha} \big)_r \Big) \leqslant {\rm Sz}_{(\eps+\rho)/2} \Big( \big(\textstyle  \bigoplus_{\alpha\in\eff}{T_\alpha} \big)_r \Big) < \omega^{\nu_\rho +\gamma} \, .
\end{align}
Let $D$ denote the domain of $T_{\gamma}$ and let $K=T_{\gamma}\dual \cball{D\dual}$, so that $s_{(\eps+\rho)/2}^{\omega^{\nu_\rho +\gamma}}(U_\eff\dual K)$ is empty for every $\eff\in\nat\fin$ by (\ref{postdoc10}) (here $U_\eff$ denotes the canonical embedding of the $\ell_r$-direct sum of the domains of the operators $T_\alpha$, $\alpha\in\eff$, into the $\ell_r$-direct sum of the domains of the operators $T_\alpha$, $\alpha<\gamma$). It follows then by an application of Lemma~\ref{postdoc2} with $\delta = (\eps+\rho)/2$ and $\eta_\delta = \omega^{\nu_\rho +\gamma}$ that
\begin{align}\label{postdoc8}
\forall \eps > \rho& >0 \quad {\rm Sz}_\eps (T_{\gamma})\leqslant \omega^{\nu_\rho +\gamma } \cdot \sigma \big(\Vert T \Vert, \, \eps, \, (\eps +\rho)/2, \, r(r-1)^{-1}\big)\, .
\end{align}

With $\pi (\eps) \in (0, \, \eps )$ as above, taking $\rho=\pi(\eps)$ in (\ref{postdoc8}) yields
\[
\forall \eps >0 \quad {\rm Sz}_\eps ( T_{\gamma} ) < \omega^{\nu_\eps +\gamma} \cdot \sigma \big(\Vert T \Vert, \, \eps, \, (\eps +\pi(\eps))/2, \, r(r-1)^{-1}\big) < \omega^{\nu_\eps +\gamma+1}\, .
\]
This concludes the inductive proof of (\ref{warnchood}).

To complete the proof of Proposition~\ref{babysitter}, we now only need show how part (ii) follows from (\ref{warnchood}). On the one hand, it is clear from the construction that $T_\alpha$ factors $T$ for each ordinal $\alpha$, hence $\mszlenk{T_\alpha} \geqslant \mszlenk{T}$. On the other hand, if $\alpha <\omega^{\zeta_T}$ then by (\ref{warnchood}) and the fact that $\nu + \omega^{\zeta_T} \leqslant \alpha_T$ whenever $\nu < \alpha_T$,
\begin{equation*}
\mszlenk{T_\alpha} = \sup_{\eps >0}\meeszlenk{\eps }{T_\alpha} \leqslant \sup_{\eps >0}\omega^{\nu_\eps +\alpha +1} \leqslant \sup_{\eps >0}\omega^{\nu_\eps +\omega^{\zeta_T}} \leqslant \omega^{\alpha_T} = \mszlenk{T} \, . \qed
\end{equation*}

\begin{remark}\label{julyjulyjuly}
It is now easy to determine precisely the Szlenk index of the operators $T_\alpha$ in terms of $\alpha$ and $\alpha_T$. Indeed, if $T$ is a noncompact Asplund operator and $\alpha$ an ordinal, then the Szlenk index of $T_\alpha$ is given by the equation
\begin{equation}\label{barbie}
\mszlenk{T_\alpha} =
\begin{cases}
\omega^{\alpha_T}& \text{if $\alpha < \omega^{\zeta_T}$},\\
\omega^{\alpha_T \, + \, (-\omega^{\zeta_T} +\alpha) \, + \, 1}& \text{if $\alpha \geqslant\omega^{\zeta_T}$,}
\end{cases}
\end{equation}
where $-\omega^{\zeta_T} + \alpha$ denotes the unique ordinal order isomorphic to $\alpha \setminus \omega^{\zeta_T}$. To prove (\ref{barbie}), one proceeds via tranfinite induction, making use of the following fact: for a set $\Lambda$, a family of Asplund operators $\set{S_\lambda \mid \lambda \in \Lambda}$ with $\sup_{\lambda \in \Lambda}\norm{S_\lambda}<\infty$, $\beta$ an ordinal such that $\sup_{\lambda \in \Lambda}\mszlenk{S_\lambda} \leqslant \omega^\beta$ and $p=0$ or $1<p<\infty$, we have ${\rm Sz} \big(( \bigoplus_{\lambda \in \Lambda}S_\lambda)_p \big) \leqslant \omega^{\beta+1}$. This fact follows easily from Proposition~\ref{unifac}, also from the results of \cite{Brookera}. Similar arguments show that if Construction~\ref{szlenkvar} is applied to a nonzero \emph{compact} operator $T$, then for all $\alpha >0$ the Szlenk index of $T_\alpha$ is $\omega^{(-1+\alpha)+1}$, where $-1 +\alpha$ denotes the unique ordinal order isomorphic to $\alpha \setminus 1$. Moreover, in this case if $\alpha>0$ is of countable cofinality and $(\alpha_n)_n$ is a non-decreasing cofinal sequence in $\alpha$, it follows from the properties of Construction~\ref{szlenkvar} discussed above that $(\bigoplus_{i=1}^nT_{\alpha_n})_1\in \szlenkop{\alpha}$ for all $n$ and $(\bigoplus_{n\in\nat}(\bigoplus_{i=1}^nT_{\alpha_n})_1)_p \notin \szlenkop{\alpha}$ ($1<p<\infty$). In particular, if $\alpha$ is of countable cofinality, then $(\szlenkop{\alpha}, \, \szlenkop{\alpha})$ is not a $\Sigma_p$-pair. \end{remark}

We require the following result from \cite[Proposition~2.18]{Brookera}:
\begin{proposition}\label{revrevrev}
Let $\alpha$ be an ordinal of countable cofinality. Then there exists an operator of Szlenk index $\omega^\alpha$.
\end{proposition}

At last, we are ready to prove Theorem~\ref{nonfactor}. For simplicity we shall assume $\beta>0$, but note that proof in the case of $\beta =0$ is achieved by similar arguments to those used here. In fact, a different proof altogether for the case $\beta = 0$ will be presented in Section~\ref{spacesect}, so there is no real loss for us in assuming $\beta$ nonzero. Moreover, there is a saving: we need not establish an analogue of Proposition~\ref{tunkyfown} for the case $\beta = 0$ (though it is not difficult to do so).

Let $\beta$ be a nonzero ordinal of countable cofinality and fix a superadditive cofinal sequence for $\omega^\beta$, which we denote $(\beta_n)_{n\in \nat}$. Since the necessary condition for membership of $\Op (\szl_{\omega^\beta})$ imposed by Proposition~\ref{tunkyfown} holds for an \emph{arbitrary} superadditive cofinal sequence for $\omega^\beta$, it suffices to construct an element of $\szl_{\omega^\beta}$ that fails this necessary condition for our fixed superadditive cofinal sequence $(\beta_n)_{n\in \nat}$. To this end, let $R$ be an operator such that $\mszlenk{R} = \omega^{\omega^\beta}$ (Proposition~\ref{revrevrev}) and note that $\mszlenk{m^{-1}R} = \omega^{\omega^\beta}$ for all $m\in\nat$. For each $m\in\nat$ let $s(m)\in\nat$ be so large that $\meeszlenk{1/2^{s(m)}}{m^{-1}R}>\omega^0 =1$, and let $W_m = (m^{-1}R)_{\beta_{s(m)^2}}$ (that is, $W_m$ is the $\beta_{s(m)^2}$th operator obtained in the application of Construction~\ref{szlenkvar} with initial operator $m^{-1}R$). Finally, set $W= (\bigoplus_{m\in\nat}W_m)_0$. To prove the theorem, we will show that $W \in \szlenkop{\omega^\beta} \setminus \Op (\szl_{\omega^\beta})$.

For each $m\in\nat$, let $E_m$ and $F_m$ denote the domain and codomain of $W_m$ respectively, so that $W \in \allop ((\bigoplus_{m\in\nat}E_m)_0, \, (\bigoplus_{m\in\nat}F_m)_0)$. Since $\beta_{s(m)^2}<\omega^\beta$ for every $m\in\nat$, it follows by Proposition~\ref{babysitter}(ii) that $W_m \in \szlenkop{\omega^\beta}$ for all $m$. For each $m\in \nat$, let $Z_m :=V_{\set{1, \ldots , m}}Q_{\set{1, \ldots , m}}W \in \szlenkop{\omega^\beta}((\bigoplus_{m\in\nat}E_m)_0, \, (\bigoplus_{m\in\nat}F_m)_0)$ (here $\set{1, \ldots , m}$ is considered a subset of the underlying index set of $(\bigoplus_{m\in\nat}F_m)_0$, and $V_{\set{1, \ldots , m}}$ and $Q_{\set{1, \ldots , m}}$ are as defined in Section~\ref{greatfamily}). Since \begin{equation*} \norm{W- Z_m} = \sup_{k>m}\norm{W_k} = (m+1)^{-1}\norm{R} \stackrel{m}{\rightarrow} 0\end{equation*} and $\szlenkop{\omega^\beta}$ is closed, it must be that $W\in \szlenkop{\omega^\beta}$.

On the other hand, by Proposition~\ref{babysitter}(i) we have that for each $m\in \nat$,
\begin{equation*}
{\rm Sz}_{1/2^{s(m)}}(W) \geqslant {\rm Sz}_{1/2^{s(m)}}(W_m) = {\rm Sz}_{1/2^{s(m)}}\Big((m^{-1}R)_{\beta_{s(m)^2}}\Big) > \omega^{\beta_{s(m)^2}} \, .
\end{equation*}
Moreover, since $\norm{m^{-1}R} \rightarrow 0$, it follows that $\set{s(m)\mid m\in\nat}$ is unbounded in $\nat$. Thus, for any $n_0\in\nat$ there is $m\in \nat$ such that $s(m) \geqslant  n_0$, and for such $m$ we have
\begin{equation*}
{\rm Sz}_{1/2^{s(m)}}(W) > \omega^{\beta_{s(m)^2}} \geqslant \omega^{\beta_{n_0 \cdot s(m)}} \, .
\end{equation*}
In particular, $W\notin \Op (\szl_{\omega^\beta})$ by Proposition~\ref{tunkyfown}. The proof of Theorem~\ref{nonfactor} is complete.

We now return to our earlier discussion regarding a universal function $\varphi$ of B.~Bossard (c.f. the paragraph following the statement of Theorem~\ref{weakfac}). The proof of Theorem~\ref{nonfactor} begins with an appeal to Proposition~\ref{revrevrev} for the existence of an operator $R$ having Szlenk index $\omega^{\omega^\beta}$. If $\beta<\omega_1$, then the construction of the operator $R$ provided by the proof of \cite[Proposition~2.18]{Brookera} (see Proposition~\ref{revrevrev} above) may be effected with the additional property that the domain and codomain of $R$ are both norm separable. Under this additional assumption, the domain and codomain of the operator $W$ constructed in the proof of Theorem~\ref{nonfactor} above are also both norm separable. Moreover, we have $\mszlenk{W} = \omega^{\omega^\beta}$ and $W \notin \Op (\szl_{\omega^\beta})$, hence $\varphi (\omega^{\omega^\beta}) > \omega^{\omega^\beta}$. Thus $\varphi$ necessarily exceeds the identity mapping of $\omega_1$ at every point of the uncountable set $\big\{\omega^{\omega^\beta}\mid \beta <\omega_1\big\}$.

\section{A class of space ideals associated with the Szlenk index}\label{spacesect}

In this section we consider a family of space ideals indexed by the class of ordinals. In particular, we shall consider the following classes, where $\alpha$ is an ordinal:
\begin{equation*}
\pzl_{\alpha}^0:= \big\{E\in \ban \,\big\vert\, \exists c\in (0, 1) \, \exists p\geqslant {1} \, \forall \, \eps \in (0, \, 1), \, s^{\omega^\alpha}_\eps (\cball{E\dual}) \subseteq (1-c\eps^p)\cball{E\dual}\big\}
\end{equation*}
and
\begin{equation*}
\pzl_{\alpha}:= \big\{E\in \ban \,\big\vert\, E\mbox{ is linearly isomorphic to some }F\in \pzl_\alpha^0\big\} \, .
\end{equation*}

The motivation for studying these classes is the following proposition, to be proved at the end of the current section.

\begin{proposition}\label{partdich}
Let $\alpha$ be an ordinal. Then at most one of the following two statements holds:

\nr{i} $\szl_{\alpha+1} = \pzl_\alpha$.

\nr{ii} $\szlenkop{\alpha+1}$ has the factorization property.
\end{proposition}

Thus, with an interest in solving Problem~\ref{factortractor}, we are prompted to ask:

\begin{question}\label{equalclasses}
Let $\alpha$ be an ordinal. Is $\pzl_\alpha = \szl_{\alpha+1}$?
\end{question}

For each ordinal $\alpha$, the inclusion $\pzl_\alpha \subseteq \szl_{\alpha+1}$ is attained via an application of Proposition~\ref{dadeeeee} with $\beta = \alpha+1$, $\beta_\eps = \omega^\alpha$ and $\delta_\eps = 1-c\eps^p$ (see also \cite[Proposition~2.2]{H'ajek2007}). The decision to consider the classes $\pzl_\alpha$ is not arbitrary, for Question~\ref{equalclasses} is known to have an affirmative answer in the case $\alpha = 0$, a result due to M. Raja \cite{Raja2009}. We thus obtain from Proposition~\ref{partdich} a proof that $\szlenkop{1}$ lacks the factorization property (Theorem~\ref{nonfactor} with $\beta=0$). We note that prior to Raja's work \cite{Raja2009}, it had been shown by H.~Knaust, E.~Odell and Th.~Schlumprecht \cite{Knaust1999} that every separable space in $\szl_1$ belongs to $\pzl_0$.

The first result to be proved in this section is the following.
\begin{proposition}\label{spaceideal} $\pzl_{\alpha}$ is a space ideal for each ordinal $\alpha$. \end{proposition}

To prove Proposition~\ref{spaceideal}, it suffices to establish the following two facts:

\nr{I} Let $E \in \pzl_{\alpha}^0$ and let $F$ be a closed linear subspace of $E$. Then $F\in \pzl_\alpha^0$.

\nr{II} Let $E, \, F \in \pzl_\alpha^0$. Then $E\oplus_\infty F \in \pzl_\alpha^0$.

The proof of (I) is straightforward. Indeed, let $i: F\hookrightarrow E$ denote the isometric linear inclusion operator and let $c' \in (0, \, 1)$ and $p'\geqslant 1$ be scalars such that $s^{\omega^\alpha}_\eps (\cball{E\dual})\subseteq (1-c'\eps^{p'})\cball{E\dual}$ for all $\eps >0$. By Lemma~\ref{poopoopoo}, for every $\eps > 0$ we have
\begin{equation}\label{windywoo}
s^{\omega^\alpha}_\eps (i\dual \cball{E\dual})\subseteq i\dual (s^{\omega^\alpha}_{\eps /2} (\cball{E\dual})) \subseteq i\dual \bigg( \Big( 1-\frac{c'}{2^{p'}}\eps^{p'}\Big) \cball{E\dual}\bigg) = \Big(1-\frac{c'}{2^{p'}}\eps^{p'}\Big)\cball{F\dual} \, .
\end{equation}
As $s^{\omega^\alpha}_\eps (\cball{F\dual})=s^{\omega^\alpha}_\eps (i\dual \cball{E\dual})$, it follows from (\ref{windywoo}) that $F$ satisfies the defining property of $\pzl_\alpha^0$ with $c = c' /2^{p'}$ and $p=p'$.

The proof of (II) is somewhat more involved. Let $c' \in (0, \, 1)$ and $p' \geqslant 1$ be such that $s^{\omega^\alpha}_\eps (\cball{E\dual})\subseteq (1-c'\eps^{p'})\cball{E\dual}$ and $s^{\omega^\alpha}_\eps (\cball{F\dual})\subseteq (1-c'\eps^{p'})\cball{F\dual}$ for all $\eps >0$. We introduce the following notation: for $\eps >0$ and $a\in [0, \, 1] \subseteq \real$, define
\begin{equation*}
A_\eps^a:= \set{(b_1, b_2)\in [0, \, 1]\times [0, \, 1] \mid ab_1 + (1-a)b_2 \geqslant \eps}.
\end{equation*}
We henceforth adhere to the following notational convention: for a $w\dual$-compact set $K$ and ordinal $\alpha$, we write $s^\alpha_0 (K) = K$. As the final step in our preparation to prove (II), we state a couple of lemmas:

\begin{lemma}\label{gulcite}
Let $E$ be a Banach space, $K_1, \ldots , K_n \subseteq E\dual$ $w\dual$-compact sets, $\eps > 0$ and $\alpha$ an ordinal. Then $s^\alpha_\eps ( \bigcup_{i=1}^nK_i ) \subseteq \bigcup_{i=1}^n s^\alpha_{\eps /2}(K_i)$.
\end{lemma}

\begin{lemma}\label{anotcite}
Let $E$ and $F$ be Banach spaces, $a\in [0, \, 1] \subseteq \real$, $\eps >0$ and $\alpha$ an ordinal. Consider $a\cball{E\dual} \times (1-a)\cball{F\dual}$ as a subset of $(E\oplus_\infty F)\dual$ and let $\delta \in (0, \, \eps)$. Then
\begin{equation*}
s^{\omega^\alpha}_\eps (a\cball{E\dual} \times (1-a)\cball{F\dual}) \subseteq \bigcup_{(b_1,\, b_2)\in A_{\delta/2}^a}as^{\omega^\alpha}_{b_1}(\cball{E\dual}) \times (1-a)s^{\omega^\alpha}_{b_2}(\cball{F\dual} ) \, .
\end{equation*} \end{lemma} The proof of Lemma~\ref{gulcite} is a straightforward transfinite induction (see, for example, \cite[Lemma~3.1]{Brookera}). Lemma~\ref{anotcite} follows immediately from \cite[Lemma~3.3]{Brookera}.

Continuing towards a proof of (II), we consider the following situation: let $l\in \nat$ and suppose that $a_1, \, a_2\in\real$ are such that $a_1+a_2\leqslant 1$. For $i=1,\, 2$ let $l_i$ denote the unique integer satisfying $l_i-1 <la_i \leqslant l_i$, so that $a_i\leqslant l_i/l$. Then $l_1+l_2 -2<l(a_1+a_2)\leqslant l$, hence $l_1+l_2\leqslant l+1$. By these considerations, and by Lemma~\ref{gulcite}, Lemma~\ref{anotcite} and the fact that $\eps /9 < \eps l/ (4l+4)$ for all $l\in\nat$, the following holds for every $\eps >0$:
\begin{align*}&s^{\omega^\alpha}_\eps  (\cball{(E\oplus_\infty F)\dual})\\ \displaybreak[0]& = s^{\omega^\alpha}_\eps \Bigg( \bigcup_{a\in [0,\, 1]} a\cball{E\dual} \times (1-a)\cball{F\dual}\Bigg)\\ \displaybreak[0]
& \subseteq \bigcap_{l\in\nat} s^{\omega^\alpha}_\eps \Bigg( \bigcup_{k=0}^{l+1}\Big( \frac{k}{l}\cball{E\dual}\times \frac{l+1-k}{l}\cball{F\dual}\Big) \Bigg)\\  \displaybreak[0]&\subseteq \bigcap_{l\in\nat}\,  \bigcup_{k=0}^{l+1}s^{\omega^\alpha}_{\eps /2} \Big( \frac{k}{l}\cball{E\dual}\times \frac{l+1-k}{l}\cball{F\dual}\Big) \\  \displaybreak[0]&= \bigcap_{l\in\nat} \bigcup_{k=0}^{l+1} \frac{l+1}{l}\, s^{\omega^\alpha}_{\eps l /(2l+2)} \Big( \frac{k}{l+1}\cball{E\dual}\times \frac{l+1-k}{l+1}\cball{F\dual}\Big) \\  \displaybreak[0]&\subseteq \bigcap_{l\in\nat} \frac{l+1}{l} \bigcup_{k=0}^{l+1}\, \bigcup_{(b_1, \, b_2)\in A_{\eps /9}^{k/(l+1)}}\Big( \frac{k}{l+1}s^{\omega^\alpha}_{b_1}(\cball{E\dual})\times \frac{l+1-k}{l+1}s^{\omega^\alpha}_{b_2}(\cball{F\dual})\Big) \\  \displaybreak[0]& \subseteq \bigcap_{l\in\nat} \frac{l+1}{l} \bigcup_{k=0}^{l+1}\, \bigcup_{(b_1, \, b_2)\in A_{\eps /9}^{k/(l+1)}} \Big( \frac{k}{l+1}(1-c'b_1^{p'})\cball{E\dual}\times \frac{l+1-k}{l+1}(1-c'b_2^{p'})\cball{F\dual}\Big) \\  \displaybreak[0]
& \subseteq \bigcap_{l\in\nat} \frac{l+1}{l} \bigcup_{k=0}^{l+1}\, \bigcup_{(b_1, \, b_2)\in A_{\eps /9}^{k/(l+1)}} \Big( 1-c' \big(\textstyle \frac{k}{l+1}b_1^{p'}+\frac{l+1-k}{l+1}b_2^{p'}\big) \Big) \cball{(E\oplus_\infty F)\dual}\\  \displaybreak[0]&\subseteq \bigcap_{l\in\nat} \frac{l+1}{l} \bigcup_{k=0}^{l+1}\, \bigcup_{(b_1, \, b_2)\in A_{\eps /9}^{k/(l+1)}} \Big( 1-c' \big(\textstyle \frac{k}{l+1}b_1+\frac{l+1-k}{l+1}b_2\big)^{p'} \Big) \cball{(E\oplus_\infty F)\dual}
\\  \displaybreak[0]& \subseteq \bigcap_{l\in\nat} \frac{l+1}{l} \bigg( 1-c' \Big(\frac{\eps}{9} \Big)^{p'} \bigg) \cball{(E\oplus_\infty F)\dual} \\  \displaybreak[0]& = \bigg( 1- \Big(\frac{c'}{9^{p'}} \Big)\eps^{p'} \bigg) \cball{(E\oplus_\infty F)\dual} \, .\end{align*}
Thus $E\oplus_\infty F$ satisfies the defining property of $\pzl_\alpha^0$ with $c= c' /9^{p'}$ and $p = p'$. This concludes the proof of (II), and it follows that $\pzl_\alpha$ is a space ideal for each ordinal $\alpha$. \qed

We now establish several preliminary results which shall be used to show that $\Op (\pzl_{\alpha})$ is never closed. In what follows, we adhere to the usual convention of denoting by $\lceil a\rceil$ the least integer greater than or equal to a given real number $a$.

\begin{proposition}\label{tool1}
Let $\alpha$ be an ordinal, $E$ and $F$ Banach spaces and $T: E\To F$ an operator. If $T\in \Op (\pzl_{\alpha})$, then there exist real scalars $c \in (0, \, 1)$, $d \geqslant 0$ and $p\geqslant 1$ such that
\begin{equation*}
\meeszlenk{1/2^n}{T} \leqslant \omega^\alpha \cdot \Bigg\lceil 1-\frac{n+d}{\log_2 (1-c2^{-np})} \Bigg\rceil
\end{equation*}
for every $n\in \nat$.
\end{proposition}

\begin{proof}
The result is trivial if $T=0$, so we assume henceforth that $T\neq 0$. As $T\in \Op (\pzl_{\alpha})$, there is a Banach space $D\in \pzl_\alpha^0$ and operators $A\in \allop (E, \, D)$ and $B\in \allop (D, \, F)$ such that $T=BA$, $\norm{A}\geqslant 1$ and $\norm{B}\leqslant 1$. By Lemma~\ref{poopoopoo}, the bound $\norm{B}\leqslant 1$ ensures that $\meeszlenk{\eps}{T}\leqslant \meeszlenk{\eps}{A} \leqslant \meeszlenk{\eps /2\norm{A}}{D}$ for every $\eps >0$.

Let $c'\in (0, \, 1)$ and $p\geqslant 1$ be such that $s^{\omega^\alpha}_\eps (\cball{D\dual})\subseteq (1-c'\eps^p)\cball{D\dual}$ for every $\eps \in (0, \, 1)$, let $c=c'(2\Vert A\Vert)^{-p}$ and let $d=2+\log_2 \norm{A}$. For each $\eps \in (0, \, 1)$ define
\begin{equation*}
l_\eps := \inf \big\{l<\omega \, \big\vert \, \meeszlenk{\eps/2\norm{A}}{D}\leqslant \omega^\alpha \cdot l\big\}
\end{equation*}
and
\begin{equation*}
m_\eps := \inf \big\{m<\omega \, \big\vert \, 4\Vert A\Vert (1-c\eps^p)^m \leqslant \eps\big\} \, .
\end{equation*}

Fix $\eps \in (0, \, 1)$. By the argument used in the proof of Proposition~\ref{dadeeeee}, for each $m<\omega$ we have
\begin{equation*}
s^{\omega^\alpha \cdot m}_{\eps /(2\norm{A})}(\cball{D\dual}) \subseteq \bigg(1-c'\left( \frac{\eps}{2\norm{A}} \right)^p \bigg)^{\negthinspace \negthinspace m} \cball{D\dual} = (1-c\eps^p)^m\cball{D\dual} \, .
\end{equation*}
In particular,
\begin{align*}
s^{\omega^\alpha \cdot (m_\eps +1)}_{\eps /(2\norm{A})}(\cball{D\dual}) \subseteq s^{\omega^\alpha \cdot m_\eps +1}_{\eps /(2\norm{A})}(\cball{D\dual}) &\subseteq s_{\eps/(2\norm{A})}\big( (1-c\eps^p)^{m_\eps} \cball{D\dual} \big)\\ &\subseteq s_{\eps/(2\norm{A})}\bigg( \frac{\eps}{4\norm{A}} \cball{D\dual} \bigg) \\ &=\emptyset\,,
\end{align*}
hence $l_\eps \leqslant m_\eps+1$. As $1-c\eps^p \in (0, \, 1)$, the definition of the logarithm yields
\begin{equation*}
m_\eps = \bigg\lceil  \log_{1-c\eps^p}\hspace{-1mm}\bigg( \frac{\eps}{4\norm{A}}\bigg) \bigg\rceil = \bigg\lceil \frac{\log_2 \eps -\log_24 -\log_2\norm{A}}{\log_2(1-c\eps^p)} \bigg\rceil = \bigg\lceil \frac{\log_2 \eps -d}{\log_2(1-c\eps^p)}\bigg\rceil .
\end{equation*}

It follows now that for each $n\in \nat$ we have
\begin{equation*}
l_{1/2^n} \leqslant 1+ \bigg\lceil\frac{\log_22^{-n}-d}{\log_2(1-c2^{-np})}\bigg\rceil = \bigg\lceil 1-\frac{n+d}{\log_2(1-c2^{-np})}\bigg\rceil \, ,
\end{equation*}
hence
\begin{equation*}
\meeszlenk{1/2^n}{T}\leqslant \meeszlenk{1/(2^{n+1}\norm{A})}{D}\leqslant \omega^\alpha \cdot l_{1/2^n} \leqslant \omega^\alpha \cdot \bigg\lceil 1-\frac{n+d}{\log_2 (1-c2^{-np})} \bigg\rceil .\qedhere
\end{equation*}
\end{proof}

\begin{proposition}\label{tool2}
Let $\alpha$ be an ordinal, $\Lambda$ a set and for each $\lambda \in \Lambda$ let $D_\lambda \in \szl_\alpha$. Then $(\bigoplus_{\lambda\in \Lambda}D_\lambda)_0\in \pzl_\alpha^0$.
\end{proposition}

\begin{proof}
Fix $\eps >0$ and suppose $x\in s^{\omega^\alpha}_\eps \big( \cball{(\bigoplus_{\lambda\in\Lambda}D_\lambda)_0\dual} \big)$. By Proposition~\ref{collection}(v), $s^{\omega^\alpha}_\eps\big( U_\eff\dual \cball{(\bigoplus_{\lambda\in\Lambda}D_\lambda)_0\dual} \big) = s^{\omega^\alpha}_\eps \big( \cball{(\bigoplus_{\lambda\in\eff}D_\lambda)_0\dual} \big) = \emptyset$ for every $\eff\in\Lambda\fin$. Applying Lemma~\ref{tvl} thus yields $\Vert{U_\eff\dual x}\Vert\leqslant 1-\eps /2$ for every $\eff\in\Lambda\fin$, hence \begin{equation*} \norm{x} = \sup_{\eff\in\Lambda\fin}\norm{U_\eff\dual x}\leqslant 1-\frac{\eps}{2}\, .\end{equation*}
As $x\in s^{\omega^\alpha}_\eps \big( \cball{(\bigoplus_{\lambda\in\Lambda}D_\lambda)_0\dual} \big)$ was arbitrary, $s^{\omega^\alpha}_\eps \big( \cball{(\bigoplus_{\lambda\in\Lambda}D_\lambda)_0\dual} \big) \subseteq (1-\eps /2)\cball{(\bigoplus_{\lambda\in\Lambda}D_\lambda)_0\dual}$. In particular, $(\bigoplus_{\lambda\in\Lambda}D_\lambda)_0$ satisfies the defining property of $\pzl_\alpha^0$ with $c=1/2$ and $p=1$.
\end{proof}

\begin{proposition}\label{tool3}
For $\alpha$ an ordinal, the class $\pzl_\alpha\setminus\szl_\alpha$ is nonempty.
\end{proposition}

\begin{proof}
Let $T= I_{c_0}$, the identity operator on $c_0$. For each ordinal $\alpha$, let $T_\alpha$ be the $\alpha$th operator given by Construction~\ref{szlenkvar} with $r=0$, and let $E_\alpha$ denote the Banach space that is the domain and codomain of $T_\alpha$ (so that $T_\alpha$ is the identity operator on $E_\alpha$). With $\eta_T = 0$ and $\zeta_T=0$ in the notation introduced in the paragraph preceding Proposition~\ref{babysitter} (since $\mszlenk{c_0}=\omega$), it follows from Proposition~\ref{babysitter}(i) that there is $\eps >0$ such that $\mszlenk{E_\alpha} = \mszlenk{T_\alpha} \geqslant {\rm Sz}_\eps (T_\alpha) >\omega^\alpha$ for all ordinals $\alpha$. We thus have $E_\alpha \notin \szl_\alpha$ for all $\alpha$, and so to complete the proof it suffices to show that $E_\alpha \in \pzl_\alpha$ for all $\alpha$. In this endeavour, we proceed by transfinite induction and recall from the paragraph following Question~\ref{equalclasses} that $\pzl_\alpha \subseteq \szl_{\alpha+1}$ for all ordinals $\alpha$.

For $\alpha=0$, we have $E_0 = c_0\in \pzl_\alpha$ by an application of Proposition~\ref{tool2} with $\Lambda =\nat$ and $D_\lambda = \field$ for all $\lambda\in\Lambda$.

Suppose that $\alpha$ is an ordinal such that $E_\beta \in \pzl_\beta$ for all $\beta <\alpha$. If $\alpha$ is a successor ordinal, say $\alpha = \zeta+1$, then since $\pzl_\zeta \subseteq \szl_{\zeta+1}$ it follows by Proposition~\ref{collection}(v) that $(\bigoplus_{i=1}^n E_\zeta)_1 \in \szl_{\zeta+1}$ for all $n\in\nat$. By Proposition~\ref{tool2}, $ E_\alpha = \big(\bigoplus_{n\in\nat}(\bigoplus_{i=1}^n E_\zeta)_1\big)_0 \in \pzl_{\zeta+1} = \pzl_{\alpha} $, as required. If $\alpha$ is a limit ordinal, then for each $\beta<\alpha$ we have $E_\beta \in \pzl_\beta \subseteq \szl_{\beta+1}\subseteq \szl_{\alpha}$, hence $E_\alpha = (\bigoplus_{\beta<\alpha}E_\beta)_0\in \pzl_\alpha$ by Proposition~\ref{tool2}. This completes the induction.
\end{proof}

\begin{theorem}\label{notnotnot}
For $\alpha$ an ordinal, the operator ideal $\Op (\pzl_\alpha)$ is not closed.
\end{theorem}

\begin{proof}
Our proof relies on ideas similar to those used to prove Theorem~\ref{nonfactor}. Let $D\in \pzl_\alpha\setminus\szl_\alpha$ (c.f. Proposition~\ref{tool3}) and let $I$ denote the identity operator of $D$. As $\pzl_\alpha$ is a space ideal, $(\bigoplus_{i=1}^mD)_1 \in \pzl_\alpha$ for all $m\in \nat$.

For each $m\in\nat$, let $s(m)\in \nat$ be so large that $\meeszlenk{1/2^{s(m)}}{m^{-1}I}>\omega^\alpha$, let \begin{equation*} t(m)= \Bigg\lceil\frac{-s(m)^2}{\log_2 (1-2^{-s(m)^2})}\Bigg\rceil \end{equation*} and let $J_m = m^{-1}\big(\bigoplus_{i=1}^{t(m)}I\big)_1\in\Op (\pzl_\alpha)$. Finally, we set $J = (\bigoplus_{m\in\nat}J_m)_0$. To prove the theorem, we will show that $J\in \overline{\Op (\pzl_\alpha)} \setminus \Op (\pzl_\alpha)$.

For each $m\in\nat$ let $H_m = \big(\bigoplus_{i=1}^{t(m)}D\big)_1$, so that $J \in \allop ((\bigoplus_{m\in\nat}H_m)_0)$. For each $m$, let $L_m = V_{\set{1, \ldots , m}}Q_{\set{1, \ldots , m}}J \in \Op (\pzl_\alpha)$ (here $\set{1, \ldots , m}$ is considered a subset of the underlying index set of $(\bigoplus_{m\in\nat}H_m)_0$, and $V_{\set{1, \ldots , m}}$ and $Q_{\set{1, \ldots , m}}$ are as defined in Section~\ref{greatfamily}). Then $ \norm{L_m -J} = \sup_{k>m}\norm{J_k} = (m+1)^{-1}\stackrel{m}{\rightarrow}0$, hence $J\in \overline{\Op (\pzl_\alpha)}$.

On the other hand, by Lemma~\ref{samsung} we have that for each $m \in \nat$,
\begin{equation*}
\meeszlenk{1/2^{s(m)}}{J} \geqslant \meeszlenk{1/2^{s(m)}}{J_m} > \omega^\alpha \cdot t(m) = \omega^\alpha \cdot \Bigg\lceil\frac{-s(m)^2}{\log_2 (1-2^{-s(m)^2})}\Bigg\rceil \, .
\end{equation*}
Moreover, since $\norm{m^{-1}I}\rightarrow 0$, it follows that $\set{s(m)\mid m\in\nat}$ is unbounded in $\nat$. Thus, for any $c\in (0, \, 1)$, $d\geqslant 0$ and $p\geqslant 1$ there is $m\in \nat$ such that \begin{equation*} \Bigg\lceil\frac{-s(m)^2}{\log_2 (1-2^{-s(m)^2})}\Bigg\rceil \geqslant  \Bigg\lceil 1-\frac{s(m)+d}{\log_2 (1-c2^{-s(m)p})} \Bigg\rceil \, ,\end{equation*} and for such $m$ we have
\begin{equation*}
\meeszlenk{1/2^{s(m)}}{J} > \omega^\alpha \cdot \Bigg\lceil 1-\frac{s(m)+d}{\log_2 (1-c2^{-s(m)p})} \Bigg\rceil\, .
\end{equation*}
We have now shown that $J$ does not satisfy the conclusion of Proposition~\ref{tool1}, hence $J\notin \Op (\pzl_\alpha)$.
\end{proof}

\begin{remark}
Earlier in this section it was mentioned that recent work of M.~Raja \cite{Raja2009}, which removed the separability hypothesis from earlier work of H.~Knaust, E.~Odell and Th.~Schlumprecht \cite{Knaust1999}, leads to a different proof of the fact that $\szlenkop{1}$ lacks the factorization property. However, it is not difficult to see that the greater generality of Raja's result is in fact not needed to establish the alternative proof. To see why this is so, let $\sep$ denote the space ideal consisting of all separable Banach spaces. Taking $D=c_0$ in the proof of Theorem~\ref{notnotnot}, one obtains an operator $J$ such that the domain of $J$ is separable and $J\in\szlenkop{1}\setminus \Op (\pzl_0)$. If it were the case that $J\in \Op (\szl_1)$, then it would follow from the separability of the domain of $J$ and the main result of \cite{Knaust1999} that $J\in \Op (\szl_1 \cap \sep) = \Op (\pzl_0 \cap \sep)\subseteq \Op (\pzl_0)$ - a contradiction. Thus $J\notin \Op (\szl_1)$, hence $\szl_1$ lacks the factorization property.
\end{remark}

We conclude our results with the following proof, promised at the beginning of the section.

\begin{proof}[Proof of Proposition~\ref{partdich}]
Trivially, $\Op (\pzl_\alpha) \subseteq \Op (\szl_{\alpha+1}) \subseteq \szlenkop{\alpha+1}$. Note that statement (i) of the proposition implies $\Op (\pzl_\alpha) = \Op (\szl_{\alpha+1})$, whilst statement (ii) of the proposition implies $\Op (\szl_{\alpha+1}) = \szlenkop{\alpha+1}$. As $ \szlenkop{\alpha+1}$ is closed and $\Op (\pzl_{\alpha})$ is not, the inclusion $\Op (\pzl_\alpha)\subseteq \szlenkop{\alpha+1}$ is strict, hence (i) and (ii) cannot both hold.
\end{proof}

\section{Concluding remarks}\label{closeone}

We have shown that the operator ideals $\szlenkop{\alpha}$ fail to have the factorization property for a large (indeed, proper) class of ordinals $\alpha$. However, we have not addressed here the possibility of the operator ideals $\szlenkop{\alpha}$ possessing some sort of approximate factorization property. Noting that $\szlenkop{\alpha}$ is closed, injective and surjective for every $\alpha$, it is worth considering whether there is some composition of the closed, injective and surjective hull procedures that yields $\szl_\alpha$ from $\Op (\szl_\alpha)$ for every ordinal $\alpha$. We give some possible examples of such compositions via the open questions below:
\begin{question}\label{fquest1}
Let $\alpha$ be an ordinal. Is $\szlenkop{\alpha} = \overline{\Op (\szl_\alpha)}$?
\end{question}
\begin{question}\label{fquest2}
Let $\alpha$ be an ordinal. Is $\szlenkop{\alpha} = \Big(\overline{\Op (\szl_\alpha)}\inj\Big)\sur$?
\end{question}
Note that the injective and surjective hull procedures commute; that is, $\big(\opideal\inj\big)\sur = \big( \opideal \sur \big)\inj$ for every operator ideal $\opideal$ (c.f. \cite[Proposition~4.7.20]{Pietsch1980}). Evidently, Corollary~\ref{uncfac} ensures that the answer to Question~\ref{fquest1} and Question~\ref{fquest2} is \emph{yes} in both cases when $\alpha$ is of uncountable cofinality. We do not know if the counterexample constructed in the proof of Theorem~\ref{nonfactor} provides a counterexample to either of the two questions above. It is well-known that in the case $\alpha=0$, the answer to Question~\ref{fquest1} is \emph{no} and the answer to Question~\ref{fquest2} is \emph{yes}. Indeed, in this case $\szlenkop{\alpha}$ is precisely the class of compact operators, whilst $\Op (\szl_\alpha)$ is the class $\finiteop$ of finite rank operators; it is well-known that $\overline{\finiteop} \subsetneq \overline{\finiteop}\inj = \compactop$. However, nothing appears to be known for Question~\ref{fquest1} and Question~\ref{fquest2} in the case that $0<\cf{\alpha}\leqslant \omega$.

Besides answering Question~\ref{equalclasses} in the affirmative, one could possibly show that the operator ideals $\szlenkop{\alpha+1}$ ($\alpha \in \ord$) lack the factorization property by following a line of inquiry such as the following. Let $\alpha$ be an ordinal and $E\in \szl_{\alpha+1}$. For each $\eps > 0$, let $ m_\eps = \inf \set{m<\omega \mid \meeszlenk{\eps}{E}< \omega^\alpha \cdot m}$ (note that $m_\eps$ exists for every $\eps$). We ask: What special properties do the numbers $m_\eps$ have? Are they submultiplicative with respect to $\eps$? Do they satisfy some other general property that ensures that the growth of the $\eps$-Szlenk indices of elements of $\Op (\szl_{\alpha+1})$ is restricted in some useful way? The straightforward homogeneity argument used by Lancien in \cite{Lancien2006} to establish the submultiplicity of the $\eps$-Szlenk indices of a given Banach space does not seem to be sufficient for a useful analysis of growth properties of the numbers $m_\eps$, so a more subtle argument is likely to be required if this direction of inquiry is to prove fruitful. 

More generally, to investigate whether $\szlenkop{\alpha}$ has the factorization property for $\alpha$ an ordinal of countable cofinality, and not of the form $\omega^\beta$ for any $\beta$, one possibility would be to consider growth properties of a family of ordinals $\alpha_\eps$, $\eps>0$, or perhaps of a (finite or infinite) sequence of ordinals $(\alpha_{\eps, \, n})_n$, defined in terms of the derivations $s_\eps^\gamma$ and depending in some way on the Cantor normal form of $\alpha$. It would also be interesting to know whether such growth conditions are sufficient for factorization through a Banach space whose Szlenk index does not exceed $\omega^\alpha$.

\section*{Acknowledgements} 
The author acknowledges the support of an ANU PhD Scholarship and thanks his doctoral advisor Dr Rick Loy and the anonymous referee for many valuable suggestions that have improved the presentation of the results. The author is grateful to Dr~Matias~Raja for making available a preprint of \cite{Raja2009}.

\vspace{20mm}

\footnotesize
\noindent Philip Brooker

\noindent Mathematical Sciences Institute

\noindent Australian National University

\noindent Canberra ACT 0200, Australia

\vphantom{pahb}

\noindent \textit{E-mail addresses}: u4163111@anu.edu.au, philip.a.h.brooker@gmail.com

\end{document}